\documentclass[a4paper,12pt,reqno]{amsart}

\usepackage{amsfonts}
\usepackage{amsmath}
\usepackage{amssymb}
\usepackage{mathrsfs}
\usepackage{nomencl}
\usepackage{delarray}
\usepackage{amsthm}
\usepackage{blkarray}

\usepackage{hhline}
\usepackage[colorlinks]{hyperref}

\renewcommand\eqref[1]{(\ref{#1})} %Need with hyperref
\numberwithin{equation}{section}
\theoremstyle{plain}

\newtheorem{thm}{Theorem}[section]
\newtheorem{prop}[thm]{Proposition}

\theoremstyle{definition}
\newtheorem{defn}[thm]{Definition}

\newtheorem{rem}[thm]{Remark}

\newcommand{\inlinemaketitle}{{\let\newpage\relax\maketitle}}
\newcommand{\be}{\begin{equation}}
\newcommand{\ee}{\end{equation}}

\def\HS{{\mathtt{HS}}}
\def\op{{\mathtt{op}}}
\def\TT{{\mathbb T}}
\def\NN{{\mathbb N}}
\def\dualSU2{\frac12\NN_0}
\def\RR{{\mathbb R}}
\def\ZZ{{\mathbb Z}}
\def\C{{\mathbb C}}
\def\Gh{{\widehat{G}}}
\def\dpi{{d_\pi}}
\def\kpi{{k_\pi}}
\def\FT{{\mathscr F}}

\def\B{{\mathcal B}}
\def\M{\mathcal{M}}
\def\Dcal{{\mathcal D}}
\def\Lap{{\Delta}}
\def\Rcal{{\mathcal R}}

\def\p#1{{\left({#1}\right)}}

\def\jp#1{{\left\langle{#1}\right\rangle}}

\def\wt#1{{\widetilde{#1}}}

\def\SU2{{\rm SU(2)}}

\DeclareMathOperator{\Tr}{Tr}

\DeclareMathOperator{\rank}{rank}
\DeclareMathOperator{\card}{card}

\begin{document}
\title[Hardy-Littlewood and Hausdorff-Young-Paley inequalities]
{Hardy-Littlewood, Hausdorff-Young-Paley inequalities, and $L^p$-$L^q$ 
Fourier multipliers 
on compact homogeneous manifolds}

\author[Rauan Akylzhanov]{Rauan Akylzhanov}
\address{Rauan Akylzhanov:
  \endgraf
  Department of Mathematics
  \endgraf
  Imperial College London
  \endgraf
  180 Queen's Gate, London SW7 2AZ
  \endgraf
  United Kingdom
  \endgraf
  {\it E-mail address} {\rm r.akylzhanov14@imperial.ac.uk}
}
  
\author[Erlan Nursultanov]{Erlan Nursultanov}
\address{Erlan Nurlustanov:
\endgraf
Department of Mathematics
\endgraf
Moscow State University, Kazakh Branch
\endgraf
and
Gumilyov Eurasian National University,
\endgraf
Astana,
Kazakhstan
  \endgraf
  {\it E-mail address} {\rm er-nurs@yandex.ru}
}

\author[Michael Ruzhansky]{Michael Ruzhansky}
\address{
  Michael Ruzhansky:
  \endgraf
  Department of Mathematics
  \endgraf
  Imperial College London
  \endgraf
  180 Queen's Gate, London SW7 2AZ
  \endgraf
  United Kingdom
  \endgraf
  {\it E-mail address} {\rm m.ruzhansky@imperial.ac.uk}
  }

\thanks{The third
 author was supported in parts by the EPSRC Grant EP/K039407/1 and by the Leverhulme Grant RPG-2014-02.}
\date{\today}


\maketitle

\begin{abstract}
	In this paper we prove new inequalities describing the relationship between  
	the ``size'' of a function on a compact homogeneous manifold and the 
	``size'' of its Fourier coefficients. 
	These inequalities can be viewed as noncommutative 
	versions of the Hardy-Littlewood inequalities obtained by Hardy and Littlewood \cite{HL}
	on the circle. For the example case of the group SU(2) we show that the obtained
	Hardy-Littlewood inequalities are sharp, yielding a criterion for a function to be in 
	$L^p(\SU2)$ in terms of its Fourier
	coefficients. We also establish Paley and Hausdorff-Young-Paley inequalities on general
	compact homogeneous manifolds. The latter is applied to obtain conditions for the $L^p$-$L^q$
	boundedness of Fourier multipliers for $1<p\leq 2\leq q<\infty$ on compact homogeneous manifolds
	as well as the $L^p$-$L^q$ boundedness of general (non-invariant) operators on compact Lie groups.
	We also record an abstract version of the Marcinkiewicz interpolation 
	theorem on totally ordered discrete sets, to be used in the proofs with different 
	Plancherel measures on the unitary duals. 
\end{abstract}

%\tableofcontents

\section{Introduction}
A fundamental problem in Fourier analysis is that of investigating the relationship between 
the  ``size'' of a function and the ``size'' of its Fourier transform. 

The aim of this paper is to give necessary conditions and sufficient conditions for the 
$L^p$-integrability of a function on an arbitrary compact homogeneous space $G/K$  by means of its Fourier coefficients. 
The obtained inequalities provide a noncommutative version of known results of this type on the circle 
$\TT$ and the real line $\RR$.

To explain this briefly, we recall that in \cite{HL}, Hardy and Littlewood have shown
that for $1<p\leq 2$ and $f\in L^p(\TT)$, the following inequality holds true:
\begin{equation}
\label{H_L_inequality-0}
	\sum_{m \in \ZZ}{(1+|m|)}^{p-2}|\widehat{f}(m)|^{p} \leq C\|f\|^p_{L^p(\TT)},
\end{equation}
arguing this to be a suitable extension of the Plancherel identity to $L^p$-spaces.
Hewitt and Ross \cite{HR} generalised this to the setting of compact abelian groups.
While we refer to Section \ref{SEC:main_results} and particularly to
Theorem \ref{Hardy_Littlewood} for more details on this, to give a flavour of our results,
our analogue for this on compact homogeneous manifolds $G/K$ of dimension $n=\dim G/K$ is
the inequality 
\begin{equation}
\label{sufficient_SU-2-0}
	\sum_{\pi\in\Gh_0} 
	\dpi 
	k^{p(\frac1p-\frac12)}_{\pi}
	\jp\pi^{n(p-2)}\|\widehat{f}(\pi)\|^p_{\HS}
\leq C
\|f\|^p_{L^p(G/K)},\quad 1<p\leq 2,
\end{equation}
which for $p=2$ gives the ordinary Plancherel identity on $G/K$,
see \eqref{plancherel}. Briefly, here $\Gh_0$ stands for class I representations of a compact Lie group $G$ with respect
to the subgroup $K$,
$\widehat{f}(\pi)\in\mathbb C^{d_\pi\times d_\pi}$ is the Fourier coefficient of $f$ at the representation $\pi$ of
degree $d_\pi$, $k_\pi$ is the number of invariant vectors of the representation $\pi$ with respect to $K$, and
$\jp\pi$ are the eigenvalues of the operator $(I-\Delta_{G/K})^{1/2}$ corresponding to $\pi$ 
for a Laplacian $\Delta_{G/K}$ on the compact
homogeneous space $G/K$.
We refer to
Theorem \ref{HL-groups-1} for this statement and to Section
\ref{SEC:notation-HL} for precise definitions.

In particular, in this paper we establish the following results, that we now  summarise and briefly discuss:
\begin{itemize}
\item {\em Hardy-Littlewood inequality}: The Hardy-Littlewood type inequality \eqref{sufficient_SU-2-0} holds on arbitrary compact homogeneous manifolds.
In particular, we can also rewrite it as
\begin{equation}
\label{sufficient_SU-2-2}
	\sum_{\pi\in\Gh_0} 
	\dpi 
	k_{\pi}
	\jp\pi^{n(p-2)}\left(\frac{\|\widehat{f}(\pi)\|_{\HS}}{\sqrt{k_{\pi}}}\right)^p
\leq C
\|f\|^p_{L^p(G/K)},\quad 1<p\leq 2,
\end{equation}
interpreting 
\begin{equation}\label{EQ:Plm}
\mu(Q)=\sum\limits_{\substack{\pi\in Q}}\dpi\kpi
\end{equation}
as the Plancherel measure on the set $\Gh_0$,
the `unitary dual' of the homogeneous manifold $G/K$, and $k_\pi$ the maximal rank of Fourier
coefficients matrices $\widehat{f}(\pi)$, so that e.g.
$\|\widehat{\delta}(\pi)\|_{\HS}=\sqrt{k_{\pi}}$ for the delta-function $\delta$ on $G/K$ and $\pi\in\Gh_0$. 

Using the Hilbert-Schmidt norms of Fourier coefficients in  \eqref{sufficient_SU-2-0}
rather than Schatten norms (leading to a different
version of $\ell^p$-spaces on the unitary dual) leads to the sharper estimate -- this is shown in
\eqref{EQ:est1} and \eqref{H_L_inequality-alt-G3}.

\item {\em Differential/Sobolev space interpretations}:
The exact form of \eqref{sufficient_SU-2-0} or \eqref{sufficient_SU-2-2} is justified in 
Section \ref{SEC:notation-HL} by comparing the differential interpretations 
\eqref{H_L_inequality-alt} and \eqref{H_L_inequality-alt-G0} 
of the classical Hardy-Littlewood inequality 
\eqref{H_L_inequality-0} and of \eqref{sufficient_SU-2-0}, respectively.
In fact, it is exactly from these differential interpretations is how
we arrive at the desired expression in \eqref{sufficient_SU-2-0}. Roughly, both are saying that
for $1<p\leq 2$, 
\begin{equation}\label{EQ:Sobolev}
g\in L^p_{2n(\frac1p-\frac12)}(G/K) \Longrightarrow \widehat{g}\in\ell^p(\Gh_0)
\end{equation}
with the corresponding norm estimate 
$\|\widehat{g}\|_{\ell^p(\Gh_0)}\leq C\|g\|_{L^p_{2n(\frac1p-\frac12)}(G/K)}$,
where $L^p_{2n(\frac1p-\frac12)}$ is the Sobolev space over $L^p$ of order $2n(\frac1p-\frac12)$,
and $\ell^p(\Gh_0)$ is an appropriately defined Lebesgue space $\ell^p$ on the unitary dual
$\Gh_0$ of representations relevant to $G/K$, with respect to the corresponding Plancherel measure.
In particular, as a special case we have the original Hardy-Littlewood inequality \eqref{H_L_inequality-0},
which can be reformulated as
$$
g\in L^p_{2(\frac1p-\frac12)}(\TT) \Longrightarrow \widehat{g}\in\ell^p(\ZZ),\quad 1<p\leq 2,
$$
see \eqref{H_L_inequality-alt}, since $\ell^p(\widehat{\TT}_0)\simeq \ell^p(\ZZ)$, and the Plancherel measure
is the counting measure on $\ZZ$ in this case.

\item {\em Duality}: By duality, the inequality \eqref{sufficient_SU-2-0}
remains true (with the reversed inequality) also for $2\leq p<\infty$.

\item {\em Sharpness}: The inequality \eqref{sufficient_SU-2-0} is sharp in the following sense: 
if the Fourier coefficients are positive and monotone (in a suitable sense), and a certain non-oscillation
condition holds, the inequality in \eqref{sufficient_SU-2-0} becomes an equivalence. In the case of
the circle $G=\mathbb T$, this was shown by Hardy and Littlewood 
(see Theorem \ref{THM:HL-criterion}) -- here, positivity and monotonicity
are understood classically, and the oscillation condition is automatically satisfied (see Remark \ref{REM:osc}). 
While we conjecture this equivalence to be
true for general compact homogeneous manifolds, we make this precise in the example of the
group $G=\SU2$.

\item {\em Paley inequality}: We propose \eqref{EQ:Paley0} as a Paley-type inequality that holds on general 
compact homogeneous manifolds. 
On one hand, our inequality \eqref{EQ:Paley0} extends H\"ormander's Paley inequality on $\RR^n$.
On the other hand, combined with the Weyl asymptotic formula for the eigenvalue counting function of elliptic
differential operators on the compact manifold $G/K$, it implies the
Hardy-Littlewood inequality \eqref{sufficient_SU-2-0} as a special case (and this is how we prove it too).

\item {\em Hausdorff-Young-Paley inequality}:
The Paley inequality \eqref{EQ:Paley0} and the Haus\-dorff-Young inequalities on $G/K$ in a suitable
scale of spaces $\ell^p(\Gh_0)$ on the unitary dual of $G/K$ imply the Hausdorff-Young-Paley inequality.
This is given in Theorem \ref{Cor:general_Paley_inequality}. 

\item {\em $L^p$-$L^q$ Fourier multipliers.}
The established Hausdorff-Young-Paley inequality becomes instrumental in obtaining $L^p$-$L^q$ Fourier multiplier
theorems on $G/K$ for indices $1<p\leq 2\leq q<2$. 
%This can be done in analogy to the case of the group SU(2) analysed in \cite{HLP} relying on  
%explicit formulae for the representations of SU(2). However, the case of general compact Lie groups
%and of compact homogeneous manifolds requires a more abstract approach, the topic that will be addressed elsewhere.
In Section \ref{SEC:multipliers} we give such results for Fourier multipliers on $G/K$:
for a Fourier multiplier $A$ 
acting by $\widehat{Af}(\pi)=\sigma_A(\pi)\widehat{f}(\pi)$
and $1<p\leq 2\leq q<2$ we have
$$
\|A\|_{L^p(G/K)\to L^q(G/K)}
\lesssim
\sup_{s>0} \left\{ s \mu(\pi\in\Gh_0: \, \|\sigma_A(\pi)\|_{\op}>s)^{\frac1p-\frac1q}\right\},
$$
where $\mu$ is the Plancherel measure as in \eqref{EQ:Plm}, 
see Theorem \ref{multiplier_upper_bound}.
Consequently, in Theorem \ref{THM:Lpq-G} we also give a general 
$L^p(G)$-$L^q(G)$ boundedness result for general (not necessarily invariant) operators $A$ on a compact Lie group
$G$ in terms of their matrix symbols $\sigma_A(x,\xi)$.
\end{itemize}

We now discuss some of these results, their relevance, and motivation behind them in more detail.

\medskip
In \cite{HL}, Hardy and Littlewood established the necessary condition for $f$ to be in $L^p(\TT)$ in terms of its Fourier 
coefficients for $1<p\leq 2$, and by duality the sufficient conditions for $f$ to be in $L^p(\TT)$ for $2\leq p<\infty$
(we recall these statements in Theorem \ref{Hardy_Littlewood}). 
We discuss how to extend these results to the noncommutative
setting of general compact homogeneous manifolds. This is done in Section \ref{SEC:notation-HL}
and in Theorem \ref{HL-groups-1}.

\medskip
On the circle, Hardy and Littlewood have shown that 
for $1<p<\infty$, if the Fourier coefficients $\widehat{f}(m)$ are monotone, then one also has the converse to
\eqref{H_L_inequality-0}, namely, 
\begin{equation}
\label{H-L-equivalence}
f\in L^p(\TT)
	\quad\textrm{ if and only if }\quad
	\sum\limits_{m\in \ZZ}(1+|m|)^{p-2}|\widehat{f}(m)|^p<\infty.
\end{equation}
To show that our Hardy-Littlewood inequalities in Theorem \ref{HL-groups-1} are sharp,
in Section \ref{SEC:criterion} we introduce the notion of `monotonicity' for sequences of matrix Fourier coefficients for
functions on $\SU2$, and in
Theorem \ref{THM:integrability-criterion-1-p-2-SU2}
we show that for $\frac32< p \leq 2$ and $G=\SU2$ the Hardy-Littlewood 
inequalities in Theorem \ref{HL-groups-1} can be also strengthened to provide a criterion:
if the Fourier coefficients of a central function $f\in L^{3/2}(\SU2)$ are `general monotone' and a
certain (natural) non-oscillation condition is satisfied, then
\begin{equation}
\label{H-L-equivalence-2}
f\in L^p(\SU2)
\quad\textrm{ if and only if }\quad
	\sum\limits_{\substack{l\in\frac12\NN_0}}
	(2l+1)^{\frac{5p}{2}-4}
	\|\widehat{f}(l)\|^p_{\HS}	
	<\infty.
\end{equation}
The equivalence in \eqref{H-L-equivalence-2} can be thought of as the analogue of 
\eqref{H-L-equivalence} on the circle: indeed,
on the circle, the mentioned non-oscillation condition is automatically satisfied, all
functions are central, and the power
$\frac{5p}{2}-4$ in \eqref{H-L-equivalence-2}
has a natural interpretation (in particular, for $p=2$, it boils down to the Plancherel
formula on SU(2), see \eqref{EQ:Planch-SU2}).

%\medskip
The restriction on $p$ to satisfy $\frac32<p<\frac52$ in Theorem \ref{THM:integrability-criterion-1-p-2-SU2} 
(and above in \eqref{H-L-equivalence-2}, but we are interested in $p\leq 2$ since $p>2$ will be covered by the dual
part of the Hardy-Littlewood inequality) 
is a particular instance of the fact that on compact simply connected semisimple Lie groups, the polyhedral Fourier partial sums of (a central function) $f$
converge to $f$ in $L^p$ if and only if $2-\frac1{s+1}<p<2+\frac1s$. Here the number $s$ depends on the root system
$\Rcal$ of the compact Lie group $G$ (see 
Stanton \cite{Sta1976}, Stanton and Tomas \cite{Stanton-Tomas:BAMS-1976}, and 
Colzani, Giulini and Travaglini \cite{CGT1989} for the only if statement), see Appendix \ref{APP:polyhedral} for precise
definitions and review.
It can be shown that for $G=\TT$ and $G=\SU2$, we have $s=0$ and $s=1$ respectively. 
Thus, Theorem \ref{THM:integrability-criterion-1-p-2-SU2} can be considered as a natural counterpart on $\SU2$ to the
criterion \eqref{H-L-equivalence} of Hardy and Littlewood on the circle.
In order to prove the above statements, we need to develop several things which are of interest on their own:
\begin{itemize}
%\item In Section \ref{SEC:DK_estimates}, in Proposition \ref{PROP:DK_estimates}, we prove
%estimates for Dirichlet kernels on compact homogeneous manifolds. The obtained lower bounds 
%rely on the Nikolskii inequality from \cite{NRT2014} and
%hold for all $0<p\leq\infty$ and arbitrary finite non-zero selection of Fourier coefficients. 
%For $2-\frac1{1+s}<p\leq \infty$ these bounds are sharp.
%However, in the interval $2-\frac1{1+s}<p<2$, our argument for getting the upper bound relies on
%the converse Hardy-Littlewood inequality from Theorem \ref{Akylzhanov_1_dual}, and is valid for
%polyhedral Fourier sums. See Proposition \ref{PROP:DK_estimates} for a precise statement.
\item In Proposition \ref{PROP:DK_estimate} we prove an estimate for the Dirichlet kernel on the group $\SU2$.
This estimate appears to be sharp because its application yields a sharp criterion for the $L^p$-integrability 
of functions on $\SU2$ in Theorem \ref{THM:integrability-criterion-1-p-2-SU2}.
\item In Appendix \ref{SEC:Marc_Interpol_Theorem}, we establish an abstract version of the Marcinkiewicz 
interpolation theorem on totally ordered discrete sets. Consequently, it is applied in proofs in the paper 
for different choices of the measure on the discrete unitary dual $\Gh$ and on the discrete set 
$\Gh_0\subset\Gh$
of class I representations of $G$.
\end{itemize}

In Section \ref{SEC:PHYP} we establish Paley-type inequalities on compact homogeneous manifolds.
Recall briefly that in \cite{Hormander:invariant-LP-Acta-1960} Lars H\"ormander has shown that 
if a positive function $\varphi\geq 0$ satisfies
\begin{equation}
\label{EQ:Horm}
	|\{\xi\in\RR^n\colon \varphi(\xi)\geq t\}|\leq\frac{C}{t}\quad\textrm{ for } t>0,
\end{equation}
then 
\begin{equation}\label{EQ:Paley-Rn}
\left(\,\,
\int\limits_{\RR^n}
\left|\widehat{u}\right|^p
\varphi^{2-p}\,d\xi
\right)^{\frac1p}
\lesssim
\|u\|_{L^p(\RR^n)},\quad 1 < p\leq 2.
\end{equation}
We note that condition \eqref{EQ:Horm} is equivalent to 
$$
M_{\varphi}:=\sup_{\substack{t>0}}	t|\{\xi\in\RR^n\colon \varphi(\xi)\geq t\}|<\infty.
$$
Our analogue for this is the inequality
\begin{equation}\label{EQ:Paley0}
\left(
	\sum\limits_{\pi\in\Gh_0}\dpi k_\pi^{p(\frac1p-\frac12)}\|\widehat{f}(\pi)\|^p_{\HS}
\ {\varphi(\pi)}^{2-p}
\right)^{\frac1p}
\lesssim
M_{\varphi}^{\frac{2-p}{p}}
\|f\|_{L^p(G/K)}, \quad 1<p\leq 2,
\end{equation}
where $\varphi(\pi)$ is a positive sequence over $\Gh_0$ such that
$$
M_{\varphi}
:=
\sup_{\substack{t>0}}t\sum\limits_{\substack{\pi\in\Gh_0\\ \varphi(\pi)\geq t}}\dpi\kpi
<\infty.
$$
Here, as well as in other results of this paper, the measure $\mu(Q)=\sum\limits_{\substack{\pi\in Q}}\dpi\kpi$
appears as an analogue of the Plancherel measure on sets $Q\subset \Gh_0$.

The sum over an empty set in the definition of $M_\varphi$ is assumed to be zero.
With $\varphi(\pi)={\jp\pi}^{-n}$, using the asymptotic formula 
for the Weyl eigenvalue counting function for the Laplacian on $G/K$ to show that $M_\varphi<\infty$,
inequality \eqref{EQ:Paley0} gives
inequality \eqref{sufficient_SU-2-0}. In this sense, the Paley inequality \eqref{EQ:Paley0} is an
extension of one of the Hardy-Littlewood inequalities.

We prove such Paley-type inequality in Theorem \ref{THM:Paley_inequality}.
Consequently, we can use the weighted interpolation between the Paley inequality and 
a suitable version of the noncommutative Hausdorff-Young inequality \eqref{H-Y} on the
homogeneous manifolds. This yields what we then call the Hausdorff-Young-Paley inequality 
in Theorem \ref{Cor:general_Paley_inequality}. This
inequality is very useful for obtaining the $L^p$-$L^q$ multiplier theorems for Fourier multipliers
on compact Lie groups and compact homogeneous spaces. This application 
is given in Section \ref{SEC:multipliers}
to provide conditions for the $L^p$-$L^q$ boundedness of Fourier multipliers 
for $p\leq q$. A special case on $\SU2$ has been done by the authors in \cite{HLP}.
For $p=q$, the Fourier multipliers have been analysed in
\cite{RuWi2013}, with the H\"ormander-Mikhlin theorem on general compact Lie groups
established in \cite{Ruzhansky-Wirth-multipliers}, extending the results for Fourier multipliers 
on $\SU2$ by Coifman-de Guzman \cite{Coifman-deGuzman:SU2-Argentina-1970}
and Coifman and Weiss \cite{Coifman-Weiss:SU2-Argentina-1970,coifman+weiss_lnm},
to the general setting of compact Lie groups.

The paper is organised as follows. In Section \ref{SEC:main_results} we fix the notation for 
the representation theory of compact Lie groups and  formulate estimates relating functions to the behaviour of their Fourier
coefficients: the version of the Hardy--Littlewood inequalities on arbitrary compact homogeneous manifold  $G/K$ and further
extensions.  
In Section \ref{SEC:criterion} we give a criterion for the $p^{th}$ power integrability of a function on $\SU2$ in terms of its Fourier coefficients.
In Section \ref{SEC:multipliers} we obtain $L^p$-$L^q$ Fourier multiplier theorem on $G/K$ and the 
$L^p$-$L^q$ boundedness theorem for general operators on $G$.
In Section \ref{SEC:proofs} we complete the proofs of the results presented in previous sections.
In  Section \ref{SEC:DK_estimate} we give an interesting estimate for the Dirichlet kernel on $\SU2$
which is instrumental in the proof of the inverse to the Hardy-Littlewood inequality on the case of the group being $\SU2$.
In Appendix \ref{APP:polyhedral} we briefly review the topic of polyhedral sums for Fourier series.
In Appendix \ref{SEC:Marc_Interpol_Theorem} we discuss a matrix-valued version of the Marcinkiewicz interpolation 
theorem that will be instrumental for our proofs.

Main inequalities in this paper are established on general compact homogeneous manifolds of the form $G/K$,
where $G$ is a compact Lie group and $K$ is a compact subgroup.
Important examples are compact Lie groups themselves when we take the trivial subgroup $K=\{e\}$
in which case $\kpi=\dpi$,
or spaces like spheres ${\mathbb S}^{n}={\rm SO}(n+1)/{\rm SO}(n)$
or complex spheres (projective spaces) $\mathbb C{\mathbb S}^{n}={\rm SU}(n+1)/{\rm SU}(n)$
in which cases the subgroups are massive and so $\kpi=1$ for all $\pi\in\Gh_{0}.$
We briefly describe such spaces and their representation theory in Section \ref{SEC:notation-HL}.
When we want to show the sharpness of the obtained inequalities, we may restrict to the
case of semisimple
Lie groups $G$. As another special case, we consider the group $\SU2$, in which case 
in Theorem \ref{THM:integrability-criterion-1-p-2-SU2} we obtain 
an analogue of the Hardy-Littlewood criterion for integrability of functions in $L^p(\SU2)$
in terms of their Fourier coefficients. This provides the converse to Hardy-Littlewood
inequalities on $\SU2$ previously obtained by the authors in \cite{HLP}.

We shall use the symbol $C$ to denote various positive constants, and $C_{p,q}$ for constants
which may depend only on indices $p$ and $q$.
We shall write $x\lesssim y$  for the relation $|x|\leq C |y|$, and write $x\cong y$ if $x\lesssim y$ and $y\lesssim x$.
\section{Main results}
\label{SEC:main_results}

In this section we introduce the necessary notation and formulate main results of the paper.
Along the exposition, we provide references to the relevant literature.

\subsection{Notation and Hardy-Littlewood inequalities}
\label{SEC:notation-HL}

In \cite[Theorems 10 and 11]{HL}, 
Hardy and Littlewood proved the following generalisation of the Plancherel's identity
on the circle $\TT$. 

\begin{thm}[Hardy--Littlewood \cite{HL}] 
\label{Hardy_Littlewood}
The following holds.
\begin{enumerate}
\item Let $1<p\leq 2$. If $f\in L^p(\TT)$, then
\begin{equation}
\label{H_L_inequality}
	\sum_{m \in \ZZ}{(1+|m|)}^{p-2}|\widehat{f}(m)|^{p} \leq C_p\|f\|^p_{L^p(\TT)},
\end{equation}
where $C_p$ is a constant which depends only on $p$.

\item  Let $2\leq p<\infty$.  
If $\{\widehat{f}(m)\}_{m\in\ZZ}$ is a sequence of complex numbers such that 
\begin{equation}
\label{H_L_condition}
\displaystyle\sum_{m\in\ZZ} (1+|m|)^{p-2}|\widehat{f}(m)|^p<\infty,
\end{equation}
then there is a function $f\in L^p(\TT)$ with Fourier coefficients given by $\widehat{f}(m)$, and 
$$
	\|f\|^p_{L^p(\TT)}\leq C_{p}^{\prime}\sum_{m\in\ZZ} (1+|m|)^{p-2}|\widehat{f}(m)|^p.
$$
\end{enumerate}
\end{thm}
Hewitt and Ross \cite{HR} generalised this theorem to the setting of compact abelian groups.
We note that if $\Delta=\partial_x^2$ is the Laplacian on $\TT$, and $\FT_\TT$ is the
Fourier transform on $\TT$, the Hardy-Littlewood inequality
\eqref{H_L_inequality} can be reformulated as
\begin{equation}
\label{H_L_inequality-alt0}
\| \FT_\TT\p{(1-\Delta)^{\frac{p-2}{2p}} f}\|_{\ell^p(\ZZ)} \leq 
   C_p\|f\|_{L^p(\TT)}.
\end{equation}
Denoting $(1-\Delta)^{\frac{p-2}{2p}} f$ by $f$ again, this becomes also equivalent to the estimate
\begin{equation}
\label{H_L_inequality-alt}
\|\widehat{f}\|_{\ell^p(\ZZ)} \leq 
   C_p\|   (1-\Delta)^{-\frac{p-2}{2p}} f\|_{L^p(\TT)}\equiv
   C_p\|   (1-\Delta)^{\frac1p-\frac12} f\|_{L^p(\TT)}, \; 1<p\leq 2.
\end{equation}
The first purpose of this section is to argue what could be a noncommutative version of these
estimates and then to establish an analogue of Theorem \ref{Hardy_Littlewood} in the setting of compact homogeneous manifolds.
To motivate the formulation, we start with a compact Lie group $G$. 
Identifying a representation $\pi$ with its equivalence class and choosing some
bases in the representation spaces, we can think of $\pi\in\Gh$ as a unitary matrix-valued
mapping $\pi:G\to\C^{\dpi\times\dpi}$. For $f\in L^{1}(G)$, we define its
Fourier transform at $\pi\in\Gh$ by
$$
(\FT_{G} f)(\pi)\equiv \widehat{f}(\pi):=\int_{G} f(u)\pi(u)^{*} du,
$$
where $du$ is the normalised Haar measure on $G$.
This definition can be extended to distributions $f\in {\mathcal D}'(G)$, and
the Fourier series takes the form
\begin{equation}\label{EQ:FS}
f(u)=\sum_{\pi\in\Gh} \dpi \Tr\p{\pi(u) \widehat{f}(\pi)}.
\end{equation}
The Plancherel identity on $G$ is given by
\begin{equation}
\label{plancherel}
\|f\|^{2}_{L^{2}(G)}=\sum_{\pi\in\Gh} \dpi \|\widehat{f}(\pi)\|_{\HS}^{2}=: 
\|\widehat{f}\|_{\ell^{2}(\Gh)}^{2},
\end{equation}
yielding the Hilbert space $\ell^{2}(\Gh)$.
Thus, Fourier coefficients of functions and distributions on $G$ take values in the space
\begin{equation}\label{EQ:Sigma}
\Sigma=\left\{ \sigma=(\sigma(\pi))_{\pi\in\Gh}: \sigma(\pi)\in \C^{\dpi\times\dpi} \right\}.
\end{equation}

The $\ell^p$-spaces on the unitary dual of a compact Lie group can be defined,
for example, motivated
by the Hausdorff--Young inequality in the form
\begin{equation}\label{EQ:HY}
\p{\sum_{\pi\in\Gh} \dpi \|\widehat{f}(\pi)\|_{S^{p'}}^{p'}}^{1/p'} \leq \|f\|_{L^p(G)}
\textrm{ for } 1<p\leq 2,
\end{equation}
with an obvious modification for $p=1$, with $\frac1p+\frac{1}{p'}=1$, and where
$S^{p'}$ is the $p'$-Schatten class on the space of matrices $\C^{\dpi\times\dpi}$.
For the inequality \eqref{EQ:HY} see \cite{Kunze:FT-TAMS-1958}.
Thus, for any $1\leq p<\infty$ we can define the 
(Schatten-based)
spaces $\ell^{p}_{sch}(\Gh)\subset\Sigma$ 
by the norm
\begin{equation}\label{EQ:Lpsch}
\|\sigma\|_{\ell^p_{sch}(\Gh)}:=\p{ \sum_{\pi\in\Gh} \dpi \|\sigma(\pi)\|_{S^p}^p}^{1/p},
\; \sigma\in\Sigma.
\end{equation}
The Hausdorff-Young inequality \eqref{EQ:HY} can be then reformulated as
$$
\|\widehat{f}\|_{\ell^{p'}_{sch}(\Gh)} \leq  \|f\|_{L^p(G)}
\textrm{ for } 1<p\leq 2.
$$
We refer to Hewitt and Ross \cite[Section 31]{Hewitt-Ross:BK-Vol-II} or to Edwards
\cite[Section 2.14]{Edwards:BK} for a thorough analysis of these spaces.

At the same time, another scale of $\ell^p$-spaces on the unitary dual $\Gh$ has been 
developed in \cite{RT} based on fixing the Hilbert-Schmidt norms, and this scale will
actually provide sharper results in our problem. 
In view of subsequently established converse estimates using the same expressions,
it appears that this scale of spaces is the correct one for extending
the Hardy-Littlewood inequalities to the noncommutative setting.
Thus, for 
$1\leq p<\infty$, we define the space
$\ell^p(\Gh)$ by the norm
\begin{equation}\label{EQ:Lp}
\|\sigma\|_{\ell^p(\Gh)}:=\p{ \sum_{\pi\in\Gh} d_\pi^{p\p{\frac2p-\frac12}} \|\sigma(\pi)\|_{\HS}^p}^{1/p},
\; \sigma\in\Sigma, \; 1\leq p<\infty,
\end{equation}
where $\|\cdot\|_{\HS}$ denotes the Hilbert-Schmidt matrix norm i.e.
$$
\|\sigma(\pi)\|_{\HS}:=\left(\Tr(\sigma(\pi)\sigma(\pi)^*)\right)^{\frac12}.
$$
It was shown in \cite[Section 10.3]{RT} that, among other things, these are interpolation spaces,
and that the Fourier transform $\FT_{G}$ and its inverse $\FT_{G}^{-1}$ satisfy the
Hausdorff-Young inequalities in these spaces.

The power of $\dpi$ in \eqref{EQ:Lp} can be naturally interpreted if we rewrite it in the form
\begin{equation}\label{EQ:Lp-2}
\|\sigma\|_{\ell^p(\Gh)}:=\p{ \sum_{\pi\in\Gh} d_\pi^2 \left(\frac{\|\sigma(\pi)\|_{\HS}}
{\sqrt{\dpi}}\right)^p}^{1/p},
\; \sigma\in\Sigma, \; 1\leq p<\infty,
\end{equation}
and think of $\mu(Q)=\sum_{\pi\in Q} d_\pi^2$ as the Plancherel measure on $\Gh$, and of
$\sqrt{\dpi}$ as the normalisation for matrices $ \sigma(\pi)\in \C^{\dpi\times\dpi}$, in view of
$\|I_\dpi\|_\HS=\sqrt{\dpi}$ for the identity matrix $I_\dpi\in \C^{\dpi\times\dpi}$.

We note that for a matrix $\sigma(\pi)\in \C^{\dpi\times\dpi}$, for $1\leq p\leq 2$, by
H\"older inequality we have
$$
\|\sigma(\pi)\|_{S^p}\leq d_\pi^{\frac1p-\frac{1}{2}} \|\sigma(\pi)\|_\HS.
$$
Consequently, for $1\leq p\leq 2$, one can show the embedding 
$\ell^p(\Gh)\subset \ell^p_{sch}(\Gh)$, with the inequality
\begin{equation}\label{EQ:norms}
\|\sigma\|_{\ell^p_{sch}(\Gh)}\leq \|\sigma\|_{\ell^p(\Gh)},\; \forall \sigma\in\Sigma, \; 1\leq p\leq 2.
\end{equation}

We now describe the setting of Fourier coefficients on a compact homogeneous
manifold $M$ following 
\cite{Dasgupta-Ruzhansky:Gevrey-BSM} or \cite{NRT2014},
and referring for further details with
proofs to Vilenkin \cite{Vilenkin:BK-eng-1968} or to
Vilenkin and Klimyk \cite{VK1991}.

Let $G$ be a compact motion group of $M$ and let $K$ be the stationary
subgroup of some point.
Alternatively, we can start with a compact Lie group $G$ with
a closed subgroup $K$, and identify $M=G/K$ as an analytic manifold in a canonical way.
We normalise measures so
that the measure on $K$ is a probability one.
Typical examples are the spheres ${\mathbb S}^{n}={\rm SO}(n+1)/{\rm SO}(n)$
or complex spheres $\mathbb C{\mathbb S}^{n}={\rm SU}(n+1)/{\rm SU}(n)$.

Let us denote by $\Gh_{0}$ the subset of $\Gh$ of representations that are
class I with respect to the subgroup $K$. This means that 
$\pi\in\Gh_{0}$ if $\pi$ has at least one non-zero invariant vector $a$ with respect
to $K$, i.e. that
$$\pi(h)a=a \; \textrm{ for all } \;h\in K.$$ 
Let $\B_{\pi}$ denote the space of these invariant vectors and let  
$$\kpi:=\dim\B_{\pi}.$$
Let us fix an orthonormal basis in the representation space of $\pi$ so that
its first $\kpi$ vectors are the basis of $B_{\pi}.$
The matrix elements $\pi(x)_{ij}$, $1\leq j\leq \kpi$,
are invariant under the right shifts by $K$.

We note that if $K=\{e\}$ so that $M=G/K=G$ is the Lie group, we have
$\Gh=\Gh_{0}$ and $\kpi=\dpi$ for all $\pi$. As the other extreme, if
$K$ is a massive subgroup of $G$, i.e., if for every such $\pi$ there is precisely one invariant vector
with respect to $K$, we have $k_{\pi}=1$ for all $\pi\in\Gh_{0}.$
This is, for example, the case for the spheres $M={\mathbb S}^{n}.$
Other examples can be found in Vilenkin \cite{Vilenkin:BK-eng-1968}.

We can now identify functions on $M=G/K$ with functions on 
$G$ which are constant on left cosets with respect to $K$. 
Then, for a function $f\in C^{\infty}(M)$ we can
recover it by the Fourier series of its canonical
lifting $\wt{f}(g):=f(gK)$ to $G$,
$\wt{f}\in C^{\infty}(G)$, and the Fourier
coefficients satisfy  $\widehat{\wt{f}}(\pi)=0$ for all representations
with $\pi\not\in\Gh_{0}$.
Also, for class I representations $\pi\in\Gh_{0}$ we have 
$\widehat{\wt{f}}(\pi)_{{ij}}=0$ for $i>\kpi$.

With this, we can write the Fourier series of $f$ (or of $\wt{f}$, but we identify these) 
in terms of
the spherical functions $\pi_{ij}$ of the representations
$\pi\in\Gh_{0}$, with respect to the subgroup
$K$. 
Namely, the Fourier series \eqref{EQ:FS} becomes
\begin{equation}\label{EQ:FSh}
f(x)=\sum_{\pi\in\Gh_{0}} \dpi \sum_{i=1}^{\dpi}\sum_{j=1}^{\kpi}
\widehat{f}(\pi)_{ji}\pi(x)_{ij}=
\sum_{\pi\in\Gh_{0}} \dpi \Tr (\widehat{f}(\pi) \pi(x)),
\end{equation}
where, in order to have the last equality, we adopt the convention 
of setting $\pi(x)_{ij}:=0$ for all $j>\kpi$, for
all $\pi\in\Gh_{0}$.
With this convention the matrix
$\pi(x)\pi(x)^{*}$ is diagonal with the first $\kpi$ diagonal entries equal to one and
others equal to zero, so that we have
\begin{equation}\label{EQ:xi-HS}
\|\pi(x)\|_{\HS}=\sqrt{k_\pi} \textrm{ for all } \pi\in\Gh_{0}, \; x\in G/K.
\end{equation}
Following \cite{Dasgupta-Ruzhansky:Gevrey-BSM},
we will say that {\it the collection of Fourier coefficients
$\{\widehat{f}(\pi)_{ij}: \pi\in\Gh, 1\leq i,j\leq \dpi\}$ is 
of class I with respect to $K$} if
$\widehat{f}(\pi)_{ij}=0$ whenever $\pi\not\in\Gh_{0}$ or
$i>\kpi.$ By the above discussion, if the collection of Fourier
coefficients is of class I with respect to $K$, then the expressions
\eqref{EQ:FS} and \eqref{EQ:FSh} coincide and yield a function
$f$ such that $f(xh)=f(h)$ for all $h\in K$, so that this function becomes
a function on the homogeneous space $G/K$. 

For the space of Fourier coefficients of class I we define the analogue of the
set $\Sigma$ in \eqref{EQ:Sigma} by
\begin{equation}\label{EQ:SigmaGK}
\Sigma(G/K):=\{\sigma:\pi\mapsto\sigma(\pi): \;
\pi\in\Gh_{0},\; \sigma(\pi)\in\C^{\dpi\times\dpi}, \; \sigma(\pi)_{ij}=0
\textrm{ for } i>\kpi\}.
\end{equation}
In analogy to \eqref{EQ:Lp},
we can define the Lebesgue spaces $\ell^{p}(\Gh_{0})$
by the following norms which we will apply to Fourier coefficients
$\widehat{f}\in\Sigma(G/K)$ of $f\in\Dcal'(G/K)$.
Thus, for $\sigma\in\Sigma(G/K)$ we set
\begin{equation}\label{EQ:Lp-sigmaGK}
\|\sigma\|_{\ell^{p}(\Gh_{0})}:=\left(\sum_{\pi\in\Gh_{0}} \dpi
k_{\pi}^{p(\frac{1}{p}-\frac12)}
\|\sigma(\pi)\|_{\HS}^{p}\right)^{1/p},\; 1\leq p<\infty.
\end{equation}
In the case $K=\{e\}$, so that $G/K=G$, these spaces coincide with
those defined by \eqref{EQ:Lp} since $\kpi=\dpi$ in this case.
Again, by the same argument as that in \cite{RT},
these spaces are interpolation spaces and the Hausdorff-Young
inequality holds for them. We refer to
\cite{NRT2014} for some more details on
these spaces.

Similarly to \eqref{EQ:Lp-2}, the power of $k_\pi$ in \eqref{EQ:Lp-sigmaGK} can be naturally interpreted 
if we rewrite it in the form
\begin{equation}\label{EQ:Lp-sigmaGK-2}
\|\sigma\|_{\ell^p(\Gh_0)}:=\p{ \sum_{\pi\in\Gh_0} d_\pi k_\pi \left(\frac{\|\sigma(\pi)\|_{\HS}}
{\sqrt{k_\pi}}\right)^p}^{1/p},
\; \sigma\in\Sigma(G/K), \; 1\leq p<\infty,
\end{equation}
and think of $\mu(Q)=\sum_{\pi\in Q} d_\pi k_\pi$ as the Plancherel measure on $\Gh_0$, and of
$\sqrt{\kpi}$ as the normalisation for matrices $ \sigma(\pi)\in \C^{\dpi\times\dpi}$ under the
adopted convention on their zeros in \eqref{EQ:SigmaGK}.

\medskip
Let $\Lap_{G/K}$ be the differential operator on $G/K$ obtained by the Laplacian $\Lap_{G}$
on $G$ acting on functions that are constant on right cosets of $G$, i.e.,
such that $\widetilde{\Lap_{G/K}f}=\Lap_{G}\widetilde{f}$ for $f\in C^{\infty}(G/K)$.

Recalling, that the Hardy-Littlewood inequality can be formulated as 
\eqref{H_L_inequality-alt0} 
or
\eqref{H_L_inequality-alt},
we will show that the analogue of \eqref{H_L_inequality-alt} 
on a compact homogeneous manifold $G/K$ becomes
\begin{equation}
\label{H_L_inequality-alt-G0}
\| \widehat{ f}\|_{\ell^p(\Gh_0)} \leq
   C_p\|(1-\Delta_{G/K})^{n(\frac1p-\frac12)} f\|_{L^p(G/K)},
\end{equation}
where $n=\dim G/K$. This yields sharper results compared to using the
Schatten-based space $\ell^p_{sch}(\Gh_0)$ in 
view of the inequality
$$\| \widehat{ f}\|_{\ell^p_{sch}(\Gh_0)} \leq 
\| \widehat{ f}\|_{\ell^p(\Gh_0)}.$$
For more extensive analysis and description of Laplace operators on compact Lie groups and on compact homogeneous manifolds we 
refer to e.g. \cite{Stein:BOOK-topics-Littlewood-Paley} and \cite{Pesenson:Besov-2008}, respectively.
We note that every representation $\pi(x)=(\pi_{ij}(x))^{\dpi}_{i,j=1}\in\Gh_0$ is invariant under the right shift by $K$. Therefore,
$\pi(x)_{ij}$ for all $1\leq i,j\leq \dpi$ are eigenfunctions of $\Delta_{G/K}$ 
with the same eigenvalue, and we denote by $\jp{\pi}$ the corresponding 
eigenvalue for the first order pseudo-differential operator $(1-\Delta_{G/K})^{1/2}$,
so that we have
$$
(1-\Delta_{G/K})^{1/2} \pi(x)_{ij}=\jp{\pi} \pi(x)_{ij} \;\textrm{ for all } 
1\leq i,j\leq \dpi.
$$

We now formulate the analogue of the Hardy-Littlewood Theorem \ref{Hardy_Littlewood}
on a compact homogeneous manifolds $G/K$ as the inequality \eqref{H_L_inequality-alt-G0} and its dual:
%We have written both $\ell^p_{sch}(\Gh_0)$ and $\ell^p(\Gh_0)$
%norms of $\widehat{f}$ in the inequality \eqref{H_L_inequality-alt-G0}
%to emphasise that both of them are dominated by the right hand side,
%however, the first inequality is just a special case of \eqref{EQ:norms}.
%Alternatively, it is equivalent to the analogue of \eqref{H_L_inequality-alt0}, in the form given in this
%theorem.
\begin{thm}[Hardy-Littlewood inequalities]
\label{HL-groups-1}
Let $G/K$ be a compact homogeneous manifold of dimension $n$. Then the following holds.
\begin{enumerate}
\item Let $1<p\leq 2$. If $f\in L^p(G/K)$, then
$\FT_{G/K}\p{(1-\Delta_{G/K})^{n(\frac12-\frac1p)} f}\in \ell^p(\Gh_0),$ and
\begin{equation}
\label{H_L_inequality-alt-G}
\| \FT_{G/K}\p{(1-\Delta_{G/K})^{n(\frac12-\frac1p)} f}\|_{\ell^p(\Gh_0)} \leq
   C_p\|f\|_{L^p(G/K)}.
\end{equation}
Equivalently, we can rewrite this estimate as
\begin{equation}
\label{H_L_inequality-alt-G2}
\sum_{\pi\in\Gh_0} 
\dpi
k_\pi^{p(\frac1p-\frac12)}
 \jp{\pi}^{n(p-2)} 
\|\widehat{f}(\pi)\|_{\HS}^p \leq
   C_p\|f\|_{L^p(G/K)}^p.
\end{equation}
\item  Let $2\leq p<\infty$.  
If $\{\sigma(\pi)\}_{\pi\in\Gh_0}\in\Sigma(G/K)$ is a sequence of complex matrices such that 
$\jp{\pi}^{n(p-2)}\sigma(\pi)$ is in $\ell^p(\Gh_0)$,
then there is a function $f\in L^p(G/K)$ with Fourier coefficients given by $\widehat{f}(\pi)=\sigma(\pi)$, and 
\begin{equation}
\label{H_L_inequality-alt-G-upper}
\|f\|_{L^p(G/K)}\leq C_{p}^{\prime}\|\jp{\pi}^{\frac{n(p-2)}{p}}\widehat{f}(\pi)\|_{\ell^p(\Gh_0)}.
\end{equation}
Using the definition of the norm on the right hand side we can write this as
\begin{equation}
\label{H_L_inequality-alt-G-upper2}
\|f\|_{L^p(G/K)}^p\leq C_{p}^{\prime}\sum_{\pi\in\Gh_0} 
\dpi
k_\pi^{p(\frac1p-\frac12)}
 \jp{\pi}^{n(p-2)} 
\|\widehat{f}(\pi)\|_{\HS}^p.
\end{equation}
\end{enumerate}
\end{thm}

For $p=2$, both of these statements reduce to the Plancherel identity \eqref{plancherel}.

We note that in view of the inequality \eqref{EQ:norms} the formulations
in terms of the space $\ell^p(\Gh_0)$ are sharper than if we used the space
$\ell^p_{sch}(\Gh_0)$. Indeed, for example, for $1<p\leq 2$, the inequality
\eqref{EQ:norms} means that
\begin{equation}\label{EQ:est1}
\sum_{\pi\in\Gh_0} 
\dpi \jp{\pi}^{n(p-2)} \|\widehat{f}(\pi)\|_{S^p}^p \leq
\sum_{\pi\in\Gh_0} 
\dpi
k_\pi^{p(\frac1p-\frac12)}
 \jp{\pi}^{n(p-2)} 
\|\widehat{f}(\pi)\|_{\HS}^p,
\end{equation}
which in turn implies
\begin{multline}
\label{H_L_inequality-alt-G3}
\| \FT_{G/K}\p{(1-\Delta_{G/K})^{n(\frac12-\frac1p)} f}\|_{\ell^p_{sch}(\Gh_0)} \\ \leq 
\| \FT_{G/K}\p{(1-\Delta_{G/K})^{n(\frac12-\frac1p)} f}\|_{\ell^p(\Gh_0)} \leq
   C_p\|f\|_{L^p(G/K)}.
\end{multline}

%Theorem \ref{HL-groups-1} follows as a special case (with $K=\{e\}$)
%from our main result
%Theorem \ref{HL-groups} which we will formulate in the setting of 
%compact homogeneous manifolds $G/K$.

\subsection{Paley and Hausdorff-Young-Paley inequalities}
\label{SEC:PHYP}

In \cite{Hormander:invariant-LP-Acta-1960},
Lars Ho\"rman\-der proved a Paley-type inequality for the Fourier transform on $\RR^n$, see \eqref{EQ:Paley-Rn}. 
Here we give an analogue of this inequality on compact homogeneous manifolds.
\begin{thm}[Paley-type inequality]\label{THM:Paley_inequality} 
Let $G/K$ be a compact homogeneous manifold. 
Let $1<p\leq 2$. If  $\varphi(\pi)$ 
is a positive sequence over $\Gh_0$
 such that
	 \begin{equation}
 \label{EQ:weak_symbol_estimate}
	 M_\varphi:=\sup_{t>0}t\sum\limits_{\substack{\pi\in\Gh_0\\ \varphi(\pi)\geq t }}\dpi\kpi<\infty
	 \end{equation}
is finite, then we have 
\begin{equation}
\label{EQ:Paley_inequality}
\left(
\sum\limits_{\pi\in\Gh_0}\dpi k_\pi^{p(\frac1p-\frac12)}\|\widehat{f}(\pi)\|^p_{\HS}\
{\varphi(\pi)}^{2-p}
\right)^{\frac1p}
\lesssim
M_\varphi^{\frac{2-p}p}
\|f\|_{L^p(G/K)}.
\end{equation}
\end{thm}
As usual, the sum over an empty set in \eqref{EQ:weak_symbol_estimate} is assumed to be zero.

With $\varphi(\pi)={\jp\pi}^{-n}$, where $n=\dim G/K$,
using the asymptotic formula \eqref{EQ:Weyl}
for the Weyl eigenvalue counting function,
we recover the first part of Theorem \ref{HL-groups-1}
(see the proof of Theorem \ref{HL-groups-1}). In this sense, the Paley inequality is an
extension of one of the Hardy-Littlewood inequalities.

Now we recall the Hausdorff-Young inequality:
\begin{eqnarray}
\label{H-Y}
\left(
\sum\limits_{\pi\in\Gh_0}
\dpi k_\pi^{{p'}(\frac1{p'}-\frac12)}
\|\widehat{f}(\pi)\|^{p'}_{\HS}
\right)^{\frac1{p'}}
\equiv
\|\widehat{f}\|_{\ell^{p'}(\Gh_0)}\lesssim \|f\|_{L^p(G/K)},\quad 1 \leq p \leq 2,
\end{eqnarray}
where, as usual, $\frac1p+\frac{1}{p'}=1$.
The inequality \eqref{H-Y} was argued in \cite{NRT2014} in analogy to \cite[Section 10.3]{RT}, so we refer there
for its justification.
Further, we recall a result on the interpolation of weighted spaces from \cite{BL2011}:
\begin{thm}[Interpolation of weighted spaces]
\label{THM:L_p-weighted-interpolation}
 Let  $d\mu_0(x)=\omega_0(x)d\mu(x),\\ d\mu_1(x)=\omega_1(x)d\mu(x)$, and write $L^p(\omega)=L^p(\omega d\mu)$ for the weight $\omega$. Suppose that $0<p_0,p_1<\infty$. Then 
$$
	(L^{p_0}(\omega_0), L^{p_1}(\omega_1))_{\theta,p}=L^p(\omega),
$$
where $0<\theta<1,\frac1p=\frac{1-\theta}{p_0}+\frac{\theta}{p_1}$, and 
$\omega=\omega^{p\frac{1-\theta}{p_0}}_0 \omega^{p\frac{\theta}{p_1}}_1$.
\end{thm}
From this, interpolating between the Paley-type inequality \eqref{EQ:Paley_inequality}  in Theorem \ref{THM:Paley_inequality} 
and  Hausdorff-Young inequality  \eqref{H-Y}, 
we obtain:
\begin{thm}[Hausdorff-Young-Paley inequality]
	\label{Cor:general_Paley_inequality}
	Let $G/K$ be a compact homogeneous manifold. 
	Let $1<p\leq b \leq p'<\infty$. 
	If a positive sequence $\varphi(\pi)$, $\pi\in\Gh_0$, satisfies condition 
% \eqref{EQ:weak_symbol_estimate} 
	 \begin{equation}
 \label{EQ:weak_symbol_estimate2}
	 M_\varphi:=\sup_{t>0}t\sum\limits_{\substack{\pi\in\Gh_0\\ \varphi(\pi)\geq t }}\dpi\kpi<\infty,
	 \end{equation}
%is finite,	with constant $M_\varphi<\infty$, 
then we have
	\begin{equation}
	\label{EQ:general_Paley_inequality}
	\left(
	\sum\limits_{\pi\in\Gh_0}
	\dpi
	k_\pi^{b(\frac1b-\frac12)}
	\left(\|\widehat{f}(\pi)\|_{\HS} \
	{\varphi(\pi)}^{\frac1b-\frac{1}{p'}}
	\right)^b
	\right)^{\frac1b}
	\lesssim
	M_\varphi^{\frac1b-\frac1{p'}}
	\|f\|_{L^p(G/K)}.
	\end{equation}
\end{thm}

This reduces to the Hausdorff-Young inequality \eqref{H-Y} when $b=p'$ and to 
the Paley inequality in \eqref{EQ:Paley_inequality} when $b=p$.

\begin{proof}[Proof of Theorem \ref{Cor:general_Paley_inequality}]
	We consider   a  sub-linear operator $A$ which takes a function $f$ to its Fourier transform $\widehat{f}(\pi)\in\C^{d_{\pi}\times d_{\pi}}$ divided by $\sqrt{\kpi}$, i.e.
$$
	L^p(G/K)\ni f 
	\mapsto
	Af
	=
	\left\{
	\frac
	{\widehat{f}(\pi)}
	{\sqrt{\kpi}}
	\right\}_{\pi\in\Gh_0}
	\in \ell^p(\Gh_0,\omega),
$$
where the spaces $\ell^p(\Gh_0,\omega)$ is defined by the norm
\begin{eqnarray*}
\|\sigma(\pi)\|_{\ell^p(\Gh_0,\omega)}
:=
\left(
\sum\limits_{\pi\in\Gh_0}\|\sigma(\pi)\|^p_{\HS}\,\omega(\pi)
\right)^{\frac1p},
\end{eqnarray*}
and $\omega(\pi)$ is a positive scalar sequence over $\Gh_0$ to be determined. 
	Then the statement follows from Theorem \ref{THM:L_p-weighted-interpolation} if we regard 
	the left-hand sides of inequalities \eqref{EQ:Paley_inequality} and
	\eqref{H-Y} as $\|Af\|_{\ell^p(\Gh_0,\omega)}$-norms in weighted sequence spaces over $\Gh_0$, 
	with the weights given by $\omega_0(\pi)=\dpi\kpi{\varphi(\pi)}^{2-p}$ and $\omega_1(\pi)=\dpi\kpi,\,\pi\in\Gh_0$,
	respectively.
\end{proof}

\subsection{Integrability criterion for functions in terms of the matrix Fourier coefficients}
\label{SEC:criterion}

In this section we show that the results of Section \ref{SEC:notation-HL} are in general sharp,
by looking at the specific example of the group SU(2) in detail.

Imposing more conditions on matrix Fourier coefficients, we make a criterion out of the 
Hardy-Littlewood inequalities in Theorem \ref{HL-groups-1}.
In fact, here we aim at obtaining a noncommutative version of the following 
criterion that Hardy and Littlewood proved in \cite{HL}:

\begin{thm} \label{THM:HL-criterion}
Let $1<p<\infty$. Suppose $f\in L^1(\TT),\,f\sim \sum \widehat{f}_m e^{2\pi i mx},$ and its 
Fourier coefficients $\{\widehat{f}_m\}_{m\in\ZZ}$ are monotone. Then we have
	\begin{equation}
	f\in L^p(\TT)
	\end{equation}
	if and only if
	\begin{equation}
	\sum\limits_{m\in \ZZ}(1+|m|)^{p-2}|\widehat{f}_m|^p<\infty.
	\end{equation}
\end{thm}

In this section we extend this result to $G=\SU2$ and
formulate a necessary and sufficient condition for $f\in L^1(\SU2)$ to belong to $L^p(\SU2)$. 
The criterion is  given in terms of the matrix Fourier coefficients. It argues that the powers 
chosen in the Hardy-Littlewood inequality are in general sharp.

First we propose a notion of general monotonicity for a sequence of matrices extending the usual
notion of monotonicity of scalars (of scalar Fourier coefficients).

\begin{defn}\label{DEF:almost_scalar} A sequence of matrices 
$\{\sigma(\pi)\}_{\pi\in\Gh_0}\in \Sigma(G/K)$  will be called {\it almost scalar} if the following conditions hold:
	\begin{enumerate}
		\item
		\label{ITEM:normal-matrix}
		For any $\pi\in\Gh_0$ the matrix $\sigma(\pi)$ is normal.
		\item 
		\label{ITEM:eigenvalue-equivalence}
		 There are constants $C_1>0$ and $C_2>0$ such that
		for any $\pi\in\Gh_0$ we have
		$$
		C_1 \leq \frac{|\lambda_i(\pi)|}{|\lambda_j(\pi)|}\leq C_2,
		$$
		for every $\lambda_i(\pi)\neq 0$ and $\lambda_j(\pi)\neq 0$,
		where $\lambda_i(\pi)\in\C$, $i=1,\ldots,\dpi$, denote the eigenvalues of 
		$\sigma(\pi)\in\C^{d_{\pi}\times d_{\pi}}$.
	\end{enumerate}
\end{defn}

As our main interest in this subsection is the group SU(2), we specify the following definition to its setting, with the
specific notation for SU(2) explained after the following definition.
This specification is done for simplicity; the following notion of monotonicity can be naturally extended to the 
setting of general compact Lie groups as well.

\begin{defn}
	\label{DEF:matrix-monotonicity}
	A sequence of matrices $\{\sigma(\l)\}_{l\in\frac12\NN_0}$ is said to be {\it monotone } if the following conditions hold:
	\begin{enumerate}
		\item
		\label{ITEM:almost-scalar-positive}
		The sequence $\{\sigma(l)\}_{l\in\frac12\NN_0}$ is almost scalar and every matrix $\sigma(\l)$ is 
		non-negative definite.
		\item 
		\label{ITEM:trace-decrease}
		Denoting by $\sigma_l$ any non-zero eigenvalue of the almost scalar matrix $\sigma(l)\in\C^{(2l+1)\times (2l+1)}$,
		the sequence $(2l+1)\sigma_l$ is decreasing, i.e.
 \begin{equation}\label{EQ:decay}
(2l+1)\sigma_{l}-(2l+2)\sigma_{l+\frac12}\geq 0
\end{equation}
 		for all $l\in\frac12\NN_0$.
\end{enumerate}
\end{defn}

In terms of general compact Lie groups, condition \eqref{EQ:decay} means that the sequence 
$\{d_\pi \sigma_\pi\}_\pi$ is decreasing along some specified ordering on the representation lattice.
In the case of the torus we have $d_\pi\equiv 1$, so this corresponds to the usual notion of
monotonicity on $\widehat{\TT}\cong\ZZ$.

\medskip
We now give a criterion for $f$ to be in $L^p$ also for $p<2$, for central functions on the compact Lie group $\SU2$.
In this case  it is common to simplify the notation, since we have the identification
of the dual $\widehat{\SU2}\cong \frac12\NN_0$ with non-negative half-integers.
Following Vilenkin \cite{Vilenkin:BK-eng-1968} it is customary to denote the representations by
$T^l\in \widehat{\SU2}$ for $l\in \frac12\NN_0$. Then we have $d_l:= d_{T^l}=2l+1$, and we abbreviate 
$\widehat{f}(T^l)=\widehat{f}(l)$.
The Plancherel identity on $\SU2$ can then be written as
\begin{equation}\label{EQ:Planch-SU2}
\|f\|^2_{L^2(\SU2)}
=
\sum\limits_{\substack{l\in\dualSU2}}
(2l+1)
\|\widehat{f}(l)\|^2_{\HS}.	
\end{equation}
We can refer to \cite{Ruzhansky+Turunen-IMRN, RT} for explicit calculations of representations and difference
operators on $\SU2$. 

\begin{rem}
The range $\frac32<p\leq 2$ appearing in Theorem \ref{THM:integrability-criterion-1-p-2-SU2}
has a natural interpretation and is related to the convergence
properties of the polyhedral Fourier partial sums. It corresponds exactly to the range 
$1<p<\infty$ on the circle. We refer to Appendix \ref{APP:polyhedral} for the detailed explanation of
these properties in terms of an auxiliary number $s$ that can be expressed in terms of the root system of the group.

In the following theorem, we denote by $L^{p}_{*}(G)$ space of central functions on $G$.
The restriction on $p$ to satisfy $2-\frac1{s+1}<p<2+\frac1s$ in the setting of compact Lie groups comes from the fact 
that on compact simply connected semisimple Lie groups, the polyhedral Fourier partial sums of (a central function) $f$
converge to $f$ in $L^p$ if and only if $2-\frac1{s+1}<p<2+\frac1s$, with $s$ defined 
as in \eqref{EQ:s} in terms of the
root system of $G$, see 
Stanton \cite{Sta1976}, Stanton and Tomas \cite{Stanton-Tomas:BAMS-1976}, and 
Colzani, Giulini and Travaglini \cite{CGT1989} for the only if statement. We recall one of such statements in 
Theorem \ref{THM:polyhedral_mean_convergence}.
In the case of $\SU2$ the number $s$  is $s=1$, so that the range
$2-\frac{1}{1+s}<p\leq 2$ that we are interested in becomes $\frac32<p\leq 2$ appearing
in Theorem \ref{THM:integrability-criterion-1-p-2-SU2}. We note that such restriction of $p>\frac32$ already appeared in
the literature on $\SU2$ also in other contexts, for example also for questions related to Fourier multipliers (extending
the results of Coifman and Weiss \cite{Coifman-Weiss:SU2-Argentina-1970}), see Clerc \cite{Clerc:SU2-CRAS-1971}.
\end{rem}

\begin{thm}
	\label{THM:integrability-criterion-1-p-2-SU2}
Let $\frac32<p\leq 2$. 
Suppose $f\in L^{3/2}_*(\SU2)$    and the sequence of its Fourier coefficients $\{\widehat{f}(l)\}_{l\in\dualSU2}$ is monotone.  Assume that there is a constant $C>0$ such that for any $\xi\in\frac12\NN_0$ the following inequality holds true 
\begin{equation}
\label{condition_on_dim}
\sum\limits_{\substack{l\in\frac12\NN_0\\ l\geq \xi}}(d_{l+\frac12}-d_{l})\widehat{f}_l\leq C d_{\xi}\widehat{f}_{\xi},
\end{equation}
where $d_l$ are the dimensions of the irreducible representations $\{T^l\}_{l\in\frac12\NN_0}$ of the group $\SU2$,
and $\widehat{f}_l$ are obtained from $\widehat{f}(l)$ as in Definition \ref{DEF:matrix-monotonicity}.
Then we have
%$f\in L^p(\SU2)$ if and only if
\begin{equation}
f\in L^p(\SU2)
\end{equation}
 if and only if 
\begin{equation}
\sum\limits_{\substack{l\in\frac12\NN_0}}
(2l+1)^{\frac{5p}2-4}
\|\widehat{f}(l)\|^p_{\HS}
<\infty.
\end{equation}
Moreover, in this case we have
\begin{equation}
\|f\|^p_{L^p(\SU2)}
\cong
\sum\limits_{\substack{l\in\dualSU2}}
(2l+1)^{\frac{5p}2-4}
\|\widehat{f}(l)\|^p_{\HS}.	
\end{equation}	
\end{thm}

\begin{rem} \label{REM:osc}
The non-oscillation type
condition \eqref{condition_on_dim} always holds for compact abelian groups, 
since in that case all the irreducible representations are $1$-dimensional, so that the expression on
the left hand side of \eqref{condition_on_dim} would be zero. Here, in the setting of SU(2), since
$d_l=2l+1$, \eqref{condition_on_dim} boils down to assuming
$$\sum\limits_{\substack{l\in\frac12\NN_0\\ l\geq \xi}}\widehat{f}_l\leq C(2\xi+1)\widehat{f}_{\xi}.$$
From this point of view this kind of assumption may be viewed as rather natural in some sense 
because it does measure how fast
the sequence of Fourier coefficients decreases compared to the dimensions of representations.
We formulate this condition in the form \eqref{condition_on_dim} to emphasise its geometric meaning:
it becomes clear how it can be extended to more general 
groups\footnote{We conjecture an analogue of Theorem \ref{THM:integrability-criterion-1-p-2-SU2} to hold
for general compact Lie groups, or even for compact homogeneous manifolds. At the moment we can not
prove it in full generality since currently we can prove in Proposition \ref{PROP:DK_estimate}
the estimate for the Dirichlet kernel that is needed for our proof
only in the setting of SU(2).}
and it is very clear that it is trivially 
satisfied on the torus.
\end{rem}

Therefore, 
Theorem \ref{THM:integrability-criterion-1-p-2-SU2} can be regarded as the direct extension of the Hardy-Littlewood
criterion in Theorem \ref{THM:HL-criterion} from the circle $\TT$ to $\SU2$. Indeed, 
the condition that functions on $\SU2$ are central is rather natural since (all) functions on $\TT$ are also
central. Moreover, the indices of $p$ correspond to each other as well, both coming from the
condition $2-\frac{1}{1+s}<p\leq 2$, which on $\TT$ becomes $1<p\leq 2$ since $s=0$, and
on $\SU2$ it is $\frac32<p\leq 2$ since $s=1$.
We also note that the assumption for functions to be central functions is rather natural since their
behaviour is very different from that of general functions, as we now briefly explain also from a 
more general perspective. 

For example, if $G$ is a compact connected semisimple Lie group and $p\not=2$, there is 
a function $f\in L^p(G)$ such that the polyhedral Fourier partial sum of $f$ does not converge to $f$ in $L^p$,
for the dilations of any open convex polyhedron in the Lie algebra of the maximal torus centred at the
origin, see Stanton and Tomas \cite{Stanton-Tomas:BAMS-1976,Stanton-Tomas:AJM-1978}. 
Such negative results are closely related
with multiplier problems for the ball and for multiple Fourier series, see
Fefferman \cite{Fefferman:divergence-BAMS-1971}. The same negative results hold also for spherical sums,
see Fefferman \cite{Fefferman:ball-AM-1971} on the torus, and on more general groups
Clerc \cite{Clerc:spherical-sums-CRAS-1972} and \cite{Clerc:PhD-thesis}.
In this paper we are using the polyhedral Fourier sums, in which case positive results become possible 
if we restrict to considering central functions. Thus, on a compact semisimple Lie group $G$,
for $2-\frac1{s+1}<p<2+\frac1s$, polyhedral Fourier partial sums of a central function $f$ converge to $f$
in $L^p(G)$, see Stanton \cite[Theorem 4.1]{Sta1976}.
If $G$ is a simple simply connected compact Lie group and $p$ falls outside of the above interval,
there are central functions in $L^p(G)$ such that their polyhedral Fourier partial sums do not converge
to $f$ in the $L^p$-norm, see Stanton and Tomas 
\cite{Stanton-Tomas:BAMS-1976,Stanton-Tomas:AJM-1978} and 
Colzani, Giulini and Travaglini \cite{CGT1989}. Such restrictions are not surprising as they 
also appear naturally in the 
multiplier problems already on $\mathbb R^n$ with $n\geq 2$: while the characteristic function of the ball is not
a multiplier on $L^p(\mathbb R^n)$ for any $p\not=2$ \cite{Fefferman:ball-AM-1971}, it does become 
an $L^p$-multiplier on radial functions if and only if $2-\frac{2}{n+1}<p<2+\frac{2}{n-1}$,
see Herz \cite{Herz:PNAS-1954}.
We refer to Appendix \ref{APP:polyhedral} for further precise statements.

%\medskip
%In the following proof we make use of estimates for the Dirichlet kernel on $\SU2$ the proof of
%which we will give in Section \ref{SEC:DK_estimate}.

\section{$L^p$-$L^q$ boundedness of operators}
\label{SEC:multipliers}

In this section we use the Hausdorff-Young-Paley inequality in Theorem \ref{Cor:general_Paley_inequality}
to give a sufficient condition for the $L^p$-$L^q$ boundedness of Fourier multipliers on compact
homogeneous spaces. It extends the condition that was obtained by a different method in \cite{Nurs_Tleukh_Mult} on the circle
$\TT$. In the case of compact Lie groups, we extend the criterion for Fourier multipliers in a rather standard way, to derive a 
condition for the  $L^p$-$L^q$ boundednes of general operators, all for the range of indices
$1<p\leq 2\leq q<\infty$.

In the case of a compact Lie group $G$, the Fourier multipliers correspond to left-invariant operators, and these
can be characterised by the condition that their symbols do not depend on the space variable. 
Thus, we can write such operators $A$ in the form
\begin{equation}\label{EQ:FM}
	\widehat{Af}(\pi)=\sigma_A(\pi)\widehat{f}(\pi),
\end{equation}
with the symbol $\sigma_A(\pi)$ depending only on $\pi\in\Gh$. The H\"ormander-Mihlin type multiplier
theorem for such operators to be bounded on $L^p(G)$ for $1<p<\infty$ was obtained in 
\cite{Ruzhansky-Wirth-multipliers}.

Now, in the context of compact homogeneous spaces $G/K$ we still want to keep the formula
\eqref{EQ:FM} as the definition of Fourier multipliers, now for all $\pi\in\Gh_0$. Indeed, due to
properties of zeros of the Fourier coefficients, we have that both sides of \eqref{EQ:FM}
are zero for $\pi\not\in\Gh_0$. Also, for $\pi\in\Gh_0$, we have
$\widehat{f}(\pi)\in\Sigma(G/K)$ with the set $\Sigma(G/K)$ defined in \eqref{EQ:SigmaGK},
which means that 
$$\widehat{f}(\pi)_{ij}=\widehat{Af}(\pi)_{ij}=0\; \textrm{ for } \; i>k_\pi.$$
Therefore, we can assume that the symbol $\sigma_A$ of a Fourier multiplier $A$ on
$G/K$ satisfies
\begin{equation}\label{EQ:FM-hom}
\sigma_A(\pi)=0 \textrm{ for } \pi\not\in\Gh_0; \textrm{ and } 
\sigma_A(\pi)_{ij}=0 \textrm{ for } \pi\in\Gh_0, \textrm{ if } i>k_{\pi} \textrm{ or } j>k_\pi.
\end{equation}
Therefore, only the upper-left block in $\sigma_A(\pi)$ of the size $k_\pi\times k_\pi$ 
may be non-zero.
Thus, we will say that $A$ is a Fourier multiplier on $G/K$ if conditions \eqref{EQ:FM} and \eqref{EQ:FM-hom}
are satisfied.

\begin{thm}
\label{multiplier_upper_bound}
Let $1<p\leq 2\leq  q < \infty$ and suppose that $A$ is a Fourier multiplier on the compact
homogeneous space $G/K$. Then we have
\begin{eqnarray}
\label{upper_estimate_refined}
\|A\|_{L^p(G/K)\to L^q(G/K)}
\lesssim
\sup_{s>0}s
\left(
\sum
\limits_{
	\substack{
		\pi\in\Gh_0\\ \|\sigma_A(\pi)\|_{\op}>s
		     }
		}\dpi\kpi
\right)^{\frac1p-\frac1q}.
\end{eqnarray}
\end{thm}

We note that if $\mu(Q)=\sum_{\pi\in Q} d_\pi k_\pi$ denotes the Plancherel measure on $\Gh_0$, then
\eqref{upper_estimate_refined} can be rewritten as
$$
\|A\|_{L^p(G/K)\to L^q(G/K)}
\lesssim
\sup_{s>0} \left\{ s \mu(\pi\in\Gh_0: \, \|\sigma_A(\pi)\|_{\op}>s)^{\frac1p-\frac1q}\right\}.
$$

\begin{rem}
Inequality \eqref{upper_estimate_refined} is sharp for $p=q=2$.	
\end{rem}
\begin{proof} 
First, we have the estimate
$$
	\|A\|_{L^2(G/K)\to L^2(G/K)}
	\leq \sup_{\pi\in\Gh_0}\|\sigma_A(\pi)\|_{\op}.
$$
Since the set
$$
\{\pi\in\Gh_0\colon \|\sigma_A(\pi)\|_{\op}\geq s\}
$$
is empty for $s>\|A\|_{L^2(G/K)\to L^2(G/K)}$ and a sum over the empty set is set to be zero, 
we have by \eqref{upper_estimate_refined}
\begin{multline*}
\|A\|_{L^2(G/K)\to L^2(G/K)}
%=\sup_{\pi\in\Gh_0}\|\sigma_A(\pi)\|_{\op}
\leq
\sup_{s>0}s \left(\sum\limits_{\substack{\pi\in\Gh_0 \\ \|\sigma_A(\pi)\|_{\op}\geq s}}\dpi\kpi\right)^0
\\=\sup_{0<s\leq \|A\|_{L^2(G/K)\to L^2(G/K)}} s\cdot 1=\|A\|_{L^2(G/K)\to L^2(G/K)}.
\end{multline*}
Thus, for $p=q=2$ we attain equality in  \eqref{upper_estimate_refined}.
\end{proof}

\begin{proof}[Proof of Theorem \ref{multiplier_upper_bound}]
	
	Recall that $A$ is a Fourier multiplier on $G/K$, i.e. 
	$$
	\widehat{Af}(\pi)=\sigma_A(\pi)\widehat{f}(\pi),
	$$
    with $\sigma_A$ satisfying \eqref{EQ:FM-hom}.
    	Since the application of  \cite[p. 14, Theorem 4.2]{HLP} with $X=G/K$ and $\mu=\{\text{Haar measure on $G$}\}$ yields
	\begin{equation}
	\|A\|_{L^p(G/K)\to L^q(G/K)}
	=
	\|A^*\|_{L^{q'}(G/K)\to L^{p'}(G/K)},
	\end{equation}
	we may assume that $p\leq q'$, for otherwise we have $q'\leq(p')'=p$ and $\|\sigma_{A^*}(\pi)\|_{\op}=\|\sigma_{A}(\pi)\|_{\op}$. 
	When $f\in C^{\infty}(G/K)$ the Hausdorff-Young inequality gives, since $q'\leq 2$,
	$$
	\|Af\|_{L^q(G/K)}
	\lesssim
	\|\widehat{Af}\|_{\ell^{q'}(\Gh_0)}
	=
	\|\sigma_A\widehat{f}\|_{\ell^{q'}(\Gh_0)}.
	$$
	We set $\sigma(\pi):=\|\sigma_A(\pi)\|^r_{\op} I_{d_{\pi}}$. It is obvious that
	\begin{equation}
	\|\sigma(\pi)\|_{\op}=\|\sigma_A(\pi)\|^r_{\op}.
	\end{equation}
	%By definition, we have
	Now, we are in a position to apply the Hausdorff-Young-Paley inequality in 
	Theorem \ref{Cor:general_Paley_inequality}.	
	With $\sigma(\pi)=\|\sigma_A\|^r I_{d_{\pi}}$ and $b=q'$, the assumption of Theorem 
	\ref{Cor:general_Paley_inequality} are then satisfied and since $\frac1{q'}-\frac1{p'}=\frac1p-\frac1q=\frac1r$, we obtain
	$$
	\|\sigma_A\widehat{f}\|_{\ell^{q'}(\Gh_0)}
	\lesssim
	\left(\sup_{s>0}s\sum\limits_{\substack{\pi\in\Gh_0 \\ \|\sigma(\pi)\|_{\op}\geq s}}\dpi\kpi\right)^{\frac1r}
	\|f\|_{L^p(G/K)},\quad f\in L^p(G/K).
	$$
Further, it can be easily checked that
	\begin{multline*}
	\left(
	\sup_{s>0}s\sum\limits_{
		\substack
		{
			\pi\in\Gh_0
			\\
			\|\sigma(\pi)\|_{\op}>s
		}
	}
\dpi\kpi
	\right)^{\frac1r}
	=
	\left(
	\sup_{s>0}s\sum\limits_{\substack{\pi\in\Gh_0\\\|\sigma_A(\pi)\|^{r}_{\op}>s}}\dpi\kpi
	\right)^{\frac1r}
	=
	\left(
	\sup_{s>0}s^{r}\sum\limits_{\substack{\pi\in\Gh\\\|\sigma_A(\pi)\|_{\op}>s}}\dpi\kpi
	\right)^{\frac1r}
	\\=
	\sup_{s>0}s \left(\sum\limits_{\substack{\pi\in\Gh\\\|\sigma_A(\pi)\|_{\op}>s}}\dpi\kpi\right)^{\frac1r}.
	\end{multline*}
This completes the proof.	
\end{proof}

A standard addition to the proof of the preceding theorem extends Theorem \ref{multiplier_upper_bound} 
to the non-invariant case.
For the simplicity in the formulation and in the understanding a variant of  \eqref{EQ:FM} in the non-invariant
case, the following result is given in the context of general compact Lie groups.
To fix the notation, we note that 
according to \cite[Theorem 10.4.4]{RT}
any linear continuous operator $A$ on $C^\infty(G)$ can be written in the form
$$Af(g)
	  	=
	  	\sum\limits_{\pi\in\Gh}
	  	\dpi\Tr\left(\pi(g) \sigma_A(g,\pi)\widehat{f}(\pi)\right)$$
for a symbol $\sigma_A$ that is well-defined on $G\times\Gh$ with values
$\sigma_A(g,\pi)\in\mathbb C^{\dpi\times\dpi}.$

\begin{thm} \label{THM:Lpq-G}
Let $1<p\leq 2 \leq q< \infty$.
Suppose that $l>\frac{p}{\dim(G)}$ is an integer. 
Let $A$ be a linear continuous operator on $C^\infty(G)$. 
Then we have
\begin{equation}\label{EQ:Lpq-G1}
\|A\|_{L^p(G)\to L^q(G)}
\lesssim
\sum\limits_{|\alpha|\leq l}
\sup_{u\in G}
\sup_{s>0}
s
\left(
\sum\limits_{\substack{\pi\in\Gh\\ \|\partial^{\alpha}_{u}\sigma_A(u,\pi)\|_{\op}\geq s}}\dpi\kpi
\right)^{\frac1p-\frac1q}.
\end{equation}	
\end{thm}
In other words, if the expression on the right hand side of \eqref{EQ:Lpq-G1} is finite, the operator
$A$ extends to a bounded operator from $L^p(G)$ to $L^q(G)$. The derivatives $\partial_u^\alpha$
are derivatives with respect to a basis of left-invariant vector fields on the Lie algebra $\mathfrak g$ of
$G$.

\begin{proof}
	  Let us define
	  $$
	  	A_uf(g)
	  	:=
	  	\sum\limits_{\pi\in\Gh}
	  	\dpi\Tr\left(\pi(g) \sigma_A(u,\pi)\widehat{f}(\pi)\right)
	  $$
	  so that $A_gf(g)=Af$. Then
\begin{equation}
\|Af\|_{L^q(G)}
=
\left(
\int\limits_{G}
|Af(g)|^q\,dg
\right)^{\frac1q}
\leq
\left(
\int\limits_{G}
\sup_{u\in G}
|A_{u}f(g)|^q\,dg
\right)^{\frac1q}.
\end{equation}
By an application of the Sobolev embedding theorem we get
$$
  \sup_{u\in G} |A_u f(g)|^q \leq
  C \sum_{|\alpha|\leq l} 
  \int\limits_G 
  |\partial_u^\alpha A_u f(g)|^q
        \ dy.
$$
Therefore, using the
Fubini theorem to change the order of integration, we obtain
\begin{eqnarray*}
    \| Af \|_{L^q(G)}^q
  &\leq & C \sum_{|\alpha|\leq l} \int\limits_G \int\limits_G
        | \partial_u^\alpha A_u f(g) |^q
        \,dg\,du \\
  &\leq & C \sum_{|\alpha|\leq l} \sup_{u\in G} \int\limits_G
        | \partial_u^\alpha A_y f(g) |^p\,dg \\
  & = & C \sum_{|\alpha|\leq l} \sup_{u\in G}
        \| \partial_u^\alpha A_u f \|_{L^q(G)}^q \\
  &\leq & C \sum_{|\alpha|\leq l} \sup_{u \in G}
        \| f\mapsto {\rm Op}(\partial^{\alpha}_u\sigma_A)f\|_
        {{\mathcal L}(L^p(G)\to L^q(G))}^q
        \|f\|_{L^p(G)}^q \\
  &\lesssim & 
 \left[
  \sum\limits_{|\alpha|\leq l}
\sup_{u\in G}
\sup_{s>0}
s
\left(
\sum\limits_{\substack{\pi\in\Gh\\ \|\partial^{\alpha}_{u}\sigma_A(u,\pi)\|_{\op}\geq s}}d^2_{\pi}
\right)^{\frac1p-\frac1q}
\right]^q
          \|f\|_{L^p(G)}^q,
\end{eqnarray*}
where the last inequality holds due to 
Theorem \ref{multiplier_upper_bound}.
This completes the proof.
\end{proof}

\section{Proofs}
\label{SEC:proofs}

In this section we prove results stated in the previous section.
We start by proving the Paley inequality in Theorem \ref{THM:Paley_inequality} and then use it to deduce the Hardy-Littlewood Theorem
 \ref{HL-groups-1}.
\subsection{Proof of Theorem \ref{THM:Paley_inequality}}
\begin{proof}[Proof of Theorem \ref{THM:Paley_inequality}]

Let $\nu$ give measure $\varphi^2(\pi)\dpi k_\pi$ to the set consisting of the single point $\{\pi\}, \pi\in\Gh_0$, 
 i.e.
$$
\nu(\{\pi\}):=\varphi^2(\pi)\dpi\kpi.
$$
We define the corresponding space $L^p(\Gh_0,\nu)$, $1\leq p<\infty$, 
as the space of complex (or real) sequences
$a=\{a_{\pi}\}_{\pi\in\Gh_0}$ such that
\begin{equation}\label{EQ:Lpmu}
\|a\|_{L^p(\Gh_0,\nu)}
:=
\left(
\sum\limits_{\substack{\pi\in\Gh_0}}
|a_{\pi}|^p
\varphi^2(\pi)\dpi\kpi
\right)^{\frac1p}
<\infty.
\end{equation}
We will show that the sub-linear operator
$$
A\colon 
L^p(G/K)\ni f  \mapsto Af
=
\left\{
\frac{
\|\widehat{f}(\pi)\|_{\HS}
}{\sqrt{\kpi}
\varphi(\pi)}
\right\}_{\pi\in\Gh_0}\in L^p(\Gh_0,\nu)
$$
is well-defined and bounded from $L^p(G/K)$ to $L^p(\Gh_0,\nu)$ for $1<p\leq 2$. 
In other words, we claim that we have the estimate
\begin{multline}
\label{Paley_inequality_alt}
\|Af\|_{L^p(\Gh_0,\nu)}
=
\left(
\sum\limits_{\substack{\pi\in\Gh_0}}
\left(
\frac{\|\widehat{f}(\pi)\|_{\HS}}{\sqrt{\kpi}\varphi(\pi)}
\right)^p
\varphi^2(\pi)\dpi\kpi
\right)^{\frac1p}
\lesssim
N^{\frac{2-p}{p}}_{\varphi}
\|f\|_{L^p(G/K)},
\end{multline}
which would give \eqref{EQ:Paley_inequality}
and where we set $N_{\varphi}:=\sup_{t>0}t\sum\limits_{\substack{\pi\in\Gh_0\\ \varphi(\pi)\geq t}}\dpi\kpi$. 
We will show that $A$ is of weak type (2,2) and of weak-type (1,1). For definition and discussions we refer to Section \ref{SEC:Marc_Interpol_Theorem} where we give definitions of weak-type, formulate and prove Marcinkiewicz-type 
interpolation Theorem \ref{THM:Marc-Interpol-Lattice} to be used in the present setting.
More precisely, with the distribution function $\nu$ as in Theorem \ref{THM:Marc-Interpol-Lattice},
we show that
\begin{eqnarray}
\label{EQ:THM:Paley_inequality_weak_1}
\nu_{\Gh_0}(y;Af)
&\leq &
\left(\frac{M_2\|f\|_{L^2(G/K)}}{y}\right)^2  \quad\text{with norm } M_2 = 1,\\
\label{EQ:THM:Paley_inequality_weak_2}
\nu_{\Gh_0}(y;Af)
&\leq &
\frac{M_1\|f\|_{L^1(G/K)}}{y}  \qquad\text{with norm } M_1 = M_{\varphi}, 
\end{eqnarray}
where $\nu_{\Gh_0}$ is defined in the Appendix in \eqref{EQ:density2}.
Then \eqref{Paley_inequality_alt} would follow by Marcin\-kiewicz interpolation theorem
(Theorem \ref{THM:Marc-Interpol-Lattice} from Section \ref{SEC:Marc_Interpol_Theorem}) 
with $\Gamma=\Gh_0$ and $\delta_{\pi}=\dpi,\kappa_{\pi}=\kpi$.

Now, to show \eqref{EQ:THM:Paley_inequality_weak_1}, using Plancherel's identity \eqref{plancherel},
we get
\begin{multline*}
y^2
\nu_{\Gh_0}(y;Af)
\leq
\|Af\|^2_{L^p(\Gh_0,\nu)}
=
\sum\limits_{\substack{\pi\in \Gh_0}}
\dpi\kpi
\left(
\frac{\|\widehat{f}(\pi)\|_{\HS}}{\sqrt{\kpi}\varphi(\pi)}
\right)^2
\varphi^2(\pi)
\\=
\sum\limits_{\substack{\pi\in\Gh_0}}
d_{\pi}
\|\widehat{f}(\pi)\|^2_{\HS}
=
\|\widehat{f}\|^2_{\ell^2(\Gh_0)}
=
\|f\|^2_{L^2(G/K)}.
\end{multline*}
Thus, $A$ is of type (2,2) with norm $M_2\leq1$.
Further, we show that $A$ is of weak-type (1,1) with norm $M_1=M_{\varphi}$; more precisely, we show that
\begin{equation}
\label{weak_type}
\nu_{\Gh_0}\{\pi\in\Gh_0 \colon \frac{\|\widehat{f}(\pi)\|_{\HS}}{\sqrt{\kpi}\varphi(\pi)} > y\}
\lesssim
M_{\varphi}\,
\dfrac{\|f\|_{L^1(G/K)}}{y}.
\end{equation}
The left-hand side here is the weighted sum $\sum \varphi^2(\pi)\dpi\kpi$ taken over those $\pi\in\Gh_0$ for which $\dfrac{\|\widehat{f}(\pi)\|_{\HS}}{\sqrt{\kpi}\varphi(\pi)}>y$. From the definition of the Fourier transform  it follows that
$$
\|\widehat{f}(\pi)\|_{\HS}\leq\sqrt{\kpi} \|f\|_{L^1(G/K)}.
$$
Therefore, we have
$$
y<\frac{\|\widehat{f}(\pi)\|_{\HS}}{\sqrt{\kpi}\varphi(\pi)}
\leq
\frac{\|f\|_{L^1(G/K)}}{\varphi(\pi)}.
$$
Using this, we get
$$
\left\{
\pi\in\Gh_0\colon 
\frac{\|\widehat{f}(\pi)\|_{\HS}}{\sqrt{\kpi}\varphi(\pi)}>y
\right\}
\subset
\left\{
\pi\in\Gh_0\colon 
\frac{\|f\|_{L^1(G/K)}}{\varphi(\pi)}>y
\right\}
$$
for any $y>0$. Consequently,
$$
\nu\left\{
\pi\in\Gh_0\colon 
\frac{\|\widehat{f}(\pi)\|_{\HS}}{\sqrt{\kpi}\varphi(\pi)}>y
\right\}
\leq
\nu\left\{
\pi\in\Gh_0\colon 
\frac{\|f\|_{L^1(G/K)}}{\varphi(\pi)}>y
\right\}.
$$
Setting $v:=\frac{\|f\|_{L^1(G/K)}}{y}$, we get

%It can be easily shown from $\sigma=\{\sigma(\pi)\}_{\pi\in\Gh}\in L_{1\infty}(\Gh)$ that
\begin{equation}
\label{PI_intermed_est_1}
\nu\left\{
\pi\in\Gh_0\colon 
\frac{\|\widehat{f}(\pi)\|_{\HS}}{\sqrt{\kpi}\varphi(\pi)}>y
\right\}
\leq
\sum\limits_{\substack{\pi\in\Gh_0 \\ \varphi(\pi)\leq v}}
\varphi^2(\pi)\dpi\kpi.
\end{equation}
We claim that
\begin{equation}\label{EQ:aux1}
\sum\limits_{\substack{\pi\in\Gh_0 \\ \varphi(\pi)\leq v}}
\varphi^2(\pi)\dpi\kpi
\lesssim 
M_{\varphi}
v.
\end{equation}
In fact, we have
$$
\sum\limits_{\substack{\pi\in\Gh_0 \\ \varphi(\pi)\leq v}}
\varphi^2(\pi)\dpi\kpi
=
\sum\limits_{\substack{\pi\in\Gh_0 \\ \varphi(\pi)\leq v}}
\dpi\kpi\int\limits^{\varphi^2(\pi)}_0 d\tau.
$$
We can interchange sum and integration to get
$$
\sum\limits_{\substack{\pi\in\Gh_0 \\ \varphi(\pi)\leq v}}
\dpi\kpi\int\limits^{\varphi^2(\pi)}_0 d\tau
=
\int\limits^{v^2}_0 d\tau 
\sum
\limits_{\substack{\pi\in\Gh_0 \\ \tau^{\frac12}\leq \varphi(\pi)\leq v}}
\dpi\kpi.
$$
Further, we make a substitution $\tau=t^2$, yielding
\begin{equation*}
\int\limits^{v^2}_0 d\tau \sum\limits_{\substack{\pi\in\Gh_0 \\ 
		\tau^{\frac12}\leq \varphi(\pi)\leq v}}\dpi\kpi
=
2\int\limits^{v}_0 t\,dt 
\sum\limits_{\substack{\pi\in\Gh_0 \\ t \leq \varphi(\pi)\leq v}}\dpi\kpi
\leq
2\int\limits^{v}_0 t\,dt 
\sum\limits_{\substack{\pi\in\Gh_0 \\ t \leq \varphi(\pi)}}\dpi\kpi.
\end{equation*}
Since 
$$
t\sum\limits_{\substack{\pi\in\Gh_0 \\ t \leq \varphi(\pi) } } \dpi\kpi
\leq 
\sup_{t>0}	t\sum\limits_{\substack{\pi\in\Gh_0 \\ t \leq \varphi(\pi) } } \dpi\kpi
=M_{\varphi}
$$
is finite by the assumption that $M_{\varphi}<\infty$, we have
$$
2\int\limits^{v}_0 t\,dt 
\sum\limits_{\substack{\pi\in\Gh_0 \\ t \leq \varphi(\pi)}}\dpi\kpi
\lesssim M_{\varphi} v.
$$
This proves \eqref{EQ:aux1}.
Thus, we have proved inequalities
\eqref{EQ:THM:Paley_inequality_weak_1},
\eqref{EQ:THM:Paley_inequality_weak_2}.
Then by using the Marcinkiewicz interpolation theorem (Theorem \ref{THM:Marc-Interpol-Lattice} from 
Section \ref{SEC:Marc_Interpol_Theorem})  with $p_1=1, p_2=2$ and 
$\frac1p=1-\theta+\frac{\theta}2$ we now obtain
$$
\left(
\sum\limits_{\substack{\pi\in\Gh_0}}
\left(
\frac{\|\widehat{f}(\pi)\|_{\HS}}{\varphi(\pi)}
\right)^p
\varphi^2(\pi)\dpi\kpi
\right)^{\frac1p}
 =
\|Af\|_{L^p(\Gh_0,\mu)}
\lesssim
M^{\frac{2-p}{p}}_{\varphi}
\|f\|_{L^p(G/K)}.
$$
This completes the proof.
\end{proof}

We now prove the Hardy-Littlewood Theorem \ref{HL-groups-1}.
\subsection{Proof of Theorem \ref{HL-groups-1}}
\begin{proof}[Proof of Theorem \ref{HL-groups-1}]
	\label{PROOF:HL-groups-1}
The second part of Theorem \ref{HL-groups-1}
follows from the first by duality, so we will concentrate on proving the first part.

Denote by $N(L)$ the eigenvalue counting function of eigenvalues (counted with multiplicities) of the first order elliptic pseudo-differential operator $(I-\Delta_{G/K})^{\frac12}$  on the compact manifold $G/K$, i.e.
\begin{equation}
\label{EQ:eigenvalue_counting_function}
	N(L)
	:=\sum\limits_{\substack{\pi\in\Gh_0\\ \jp{\pi}\leq L}}\dpi\kpi.
\end{equation}

Using the eigenvalue counting function $N(L)$, we can reformulate condition \eqref{EQ:weak_symbol_estimate} for $\varphi(\pi)=\jp\pi^{-n}$ in the following form
\begin{equation}
\label{EQ:claim}
	\sup_{0<u<+\infty}u N(u^{-\frac1n})<\infty.
\end{equation}
Since $N(L)$ is a right-continuous monotone function, the set of discontinuity points on $(0,+\infty)$ is at most countable. Therefore, without loss of generality, we can assume that $\psi(u)=uN\left(\left(\frac1u\right)^{\frac1n}\right)$ is a continuous function on $(0,+\infty)$. It is clear that $\lim_{u\to +\infty}\psi(u)=0$.
Further, we use the asymptotic of the Weyl eigenvalue counting function $N(L)$ for the first order elliptic
pseudo-differential operator $(1-\Delta_{G/K})^{1/2}$ on the compact manifold
$G/K$, to get that the eigenvalue counting function $N(L)$ (see e.g. Shubin \cite{Shubin:BK-1987})
satisfies
\begin{equation}\label{EQ:Weyl}
N(L)=\sum\limits_{\substack{\pi\in\widehat{G}_{0}\\\langle\pi\rangle\leq L}}\dpi\kpi
\cong
L^n\quad \text{for large $L$.}
\end{equation}
With $L=\left(\frac1u\right)^{\frac1n}$ and $n=\dim G/K$, this implies
$$
\lim_{u\to 0}\psi(u)=\lim uN\left(\left(\frac1{u}\right)^{\frac1n}\right)=\lim_{u\to 0} u \left(\frac1{u^{\frac1n}}\right)^{n}=\lim_{u\to 0} 1 = 1.
$$
Thus, we showed that $\psi(u)$ is a bounded function on $(0,+\infty)$, or equivalently, we established \eqref{EQ:claim}.
Then, it is clear that $\varphi(\pi)=\jp\pi^{-n}$ satisfies condition \eqref{EQ:weak_symbol_estimate}. The application of the Paley inequality from Theorem \ref{THM:Paley_inequality} yields the Hardy-Littlewood inequality. This completes the proof.
\end{proof}
\subsection{Proof of Theorem \ref{THM:integrability-criterion-1-p-2-SU2}}

\begin{proof}[Proof of Theorem \ref{THM:integrability-criterion-1-p-2-SU2}]
	In  view of Theorem \ref{HL-groups-1}, it is sufficient to prove the converse inequality, i.e.
\begin{equation}
\|f\|^p_{L^p_*(\SU2)}
\lesssim
\sum\limits_{\substack{l\in\dualSU2}}
(2l+1)^{\frac{5p}2-4}
\|\widehat{f}(l)\|^p_{\HS}.	
\end{equation}
We will first prove that there is $C>0$ such that for any $\xi\in\dualSU2$ we have
\begin{equation}
\label{EQ:integrability-criterion-1-p-2-proof-aux-1}
|f(u)|
\leq C
 \frac{1}{(2\xi+1)^2}
 	\frac1{(\sin\pi \frac{t}2)^2}
 	\left|
 	\sum\limits_{\substack{l\in\dualSU2 \\ l\leq \xi}}(2l+1)\Tr\widehat{f}(l)
 	\right|,
\end{equation}
where 
\begin{equation}\label{EQ:SU2-1}
u(t,\theta,\psi)=
\left(
  \begin{array}{cc}
    \cos(\frac{\theta}2)e^{i(2\pi t+\psi)/2} & i\sin(\frac{\theta}2)e^{i(2\pi t-\psi)/2} \\ 
    i\sin(\frac{\theta}2)e^{-i(2\pi t-\psi)/2}  &\cos(\frac{\theta}2)e^{-i(2\pi t+\psi)/2}   \\ 
  \end{array}
 \right)
\end{equation}
is a parameterisation of $\SU2$,
and the coordinates $(t,\theta,\psi)$ vary in the parameter ranges
\begin{equation}\label{EQ:SU2-2}
0\leq t<1,\quad 0\leq \theta \leq \pi,\quad -2\pi\leq\psi\leq2\pi.
\end{equation}
We refer to \cite{Ruzhansky+Turunen-IMRN} or \cite{RT} for the general discussion of the Euler angles in this setting.
We also note that due to the assumption that the Fourier coefficients are monotone, they are nonnegative and decreasing, so
the modulus on the right hand side of \eqref{EQ:integrability-criterion-1-p-2-proof-aux-1} can be actually dropped.

We fix an arbitrary half-integer $\xi\in\dualSU2$ and let $k$ be any half-integer greater than $\xi$, i.e. $k\geq \xi$,
$k\in\dualSU2$. Then we have
\begin{multline}
\label{EQ:Akylzhanov_1_dual_proof_aux_3}
\left|
\sum\limits_{\substack{l\in\dualSU2\\ l\leq k}}(2l+1)\Tr[\widehat{f}(l)T^l(u)]
\right|
\leq \\
\left|
\sum\limits_{\substack{l\in\dualSU2\\ l\leq \xi}}(2l+1)\Tr[\widehat{f}(l)T^l(u)]
\right|
+
\left|
\sum\limits_{\substack{l\in\dualSU2\\ \xi < l \leq k}}(2l+1)\Tr[\widehat{f}(l)T^l(u)]
\right|.
\end{multline}
Since $\widehat{f}(k)$ is an almost scalar sequence of the Fourier coefficients, we have
$$
\Tr[\widehat{f}(k)T^k(u)]
\cong
\widehat{f}_k\Tr T^k(u).
$$
Thus
$$
\left|
\sum\limits_{\substack{l\in\dualSU2\\ l\leq \xi}}(2l+1)\Tr[\widehat{f}(l)T^l(u)]
\right|
\leq
\sum\limits_{\substack{l\in\dualSU2\\ l\leq \xi}}(2l+1)|\widehat{f}_l||\Tr T^l(u)|.
$$
Since matrices $T^l(u)$ are unitary of size $(2l+1)\times (2l+1)$, we have
$$
	\left|\Tr T^l(u)\right|
\leq (2l+1).
$$
Therefore
$$
\sum\limits_{\substack{l\in\dualSU2\\ l\leq \xi}}(2l+1)|\widehat{f}_l||\Tr T^l(u)|
\leq
\sum\limits_{\substack{l\in\dualSU2\\ l\leq \xi}}(2l+1)^2|\widehat{f}_l|.
$$
Applying the Abel transform to $\widehat{f}_l$ and $(2l+1)\Tr[T^l(u)]$ in the second term in the sum in \eqref{EQ:Akylzhanov_1_dual_proof_aux_3}, we get
$$
\sum\limits_{\substack{l\in\dualSU2 \\ \xi\leq l\leq k}}(2l+1)\widehat{f}_l\Tr[T^l(u)]
=
\sum\limits_{\substack{l\in\dualSU2 \\ \xi\leq l\leq k-\frac12}}(\widehat{f}_{l}-\widehat{f}_{l+\frac12})D_{l}(t)
+
\widehat{f}_{k}D_{k}(t)
-
\widehat{f}_{\xi}D_{\xi-\frac12}(t)
,
$$
where $D_{k}(t)=\sum\limits_{\substack{l\in\dualSU2\\l\leq k}}(2l+1)\Tr T^l(u)$.
We will now use the estimate \eqref{EQ:DK_estimate_SU2} for the Dirichlet kernel 
from Proposition \ref{PROP:DK_estimate} that we 
postpone to be proved later. Thus, we first estimate
\begin{multline*}
	\left|
	\sum\limits_{\substack{l\in\dualSU2 \\ \xi\leq l\leq k}}(2l+1)\Tr[\widehat{f}(l)T^l(u)]
	\right|
\leq
	\left|
	\sum\limits_{\substack{l\in\dualSU2\\\xi\leq l\leq k-\frac12}}(\widehat{f}_{l}-\widehat{f}_{l+\frac12})D_{l}(t)
	\right|
+
	\left|
	\widehat{f}_{k}D_{k}(t)
	\right|
\\+
	\left|
	\widehat{f}_{\xi}D_{\xi-\frac12}(t)
	\right|
\leq
	\sum\limits_{\substack{l\in\dualSU2\\ \xi\leq l\leq k-\frac12}}\left|\widehat{f}_{l}-\widehat{f}_{l+\frac12}\right|\left|D_{l}(t)\right|
+
\left|
	\widehat{f}_{k}\right| \left| D_{k}(t)
	\right|
	+
	\left|
	\widehat{f}_{\xi}\right|
	\left|
	D_{\xi-\frac12}(t)
	\right|
\end{multline*}

Using estimate \eqref{EQ:DK_estimate_SU2} for the Dirichlet kernel
and monotonicity of $(2k+1)\widehat{f}_k$ we can estimate this as

\begin{multline*}	
\lesssim
	\frac1{t^2}
\left(
	\sum\limits_{\substack{l\in\dualSU2\\ \xi\leq l\leq k-\frac12}}[(2l+2)\widehat{f}_{l}-(2l+2)\widehat{f}_{l+\frac12}]
+ 
	(2k+1)\widehat{f}_{k}
+ 
	2\xi\widehat{f}_{\xi}
\right)
\\=
	\frac1{t^2}
\left(
	\sum\limits_{\substack{l\in\dualSU2\\ \xi\leq l\leq k-\frac12}}
	[(2l+1)\widehat{f}_{l}-(2l+2)\widehat{f}_{l+\frac12}]
+
	\sum\limits_{\substack{l\in\dualSU2\\ \xi\leq l\leq k-\frac12}}
	\widehat{f}_l
+ 
	(2k+1)\widehat{f}_{k}
+ 
	2\xi\widehat{f}_{\xi}
\right)
%\\\lesssim
%	\frac1{t^2}
%\left(
%	\left|
%	\sum\limits_{\substack{l\in\dualSU2\\ \xi+\frac12\leq l\leq k-\frac12}}[(2l+1)\widehat{f}_{l}-(2l+2)\widehat{f}_{l+\frac12}]
%	\right|
%+ 
%	\sum\limits_{\substack{l\in\dualSU2\\ \xi\leq l\leq k-\frac12}}
%	\widehat{f}_l
%+
%	(2k+1)\widehat{f}_{k}
%+ 
%	2\xi\widehat{f}_{\xi}
%\right)
\\ \lesssim
	\frac1{t^2}
\left(
	(2\xi+1)\widehat{f}_{\xi}
	-
	(2k+1)\widehat{f}_{k}
+ 
	\sum\limits_{\substack{l\in\dualSU2\\ \xi\leq l\leq k-\frac12}}
	\widehat{f}_l
+
	(2k+1)\widehat{f}_{k}
+ 
	2\xi\widehat{f}_{\xi}
\right)
\\\lesssim
\frac1{t^2}
	(2\xi+1)\widehat{f}_{\xi},
\end{multline*}
where the sum in the last line is finite even as $k\to\infty$ in view of the non-oscillating assumption 
\eqref{condition_on_dim}, namely, since
$$
\sum\limits_{\substack{l\in\dualSU2\\ \xi\leq l\leq k-\frac12}}
	\widehat{f}_l\leq \sum\limits_{\substack{l\in\frac12\NN_0 \\ l\geq \xi}}(d_{l}-d_{l+1})\widehat{f}_l
	<
	(2\xi+1)\widehat{f}_{\xi}.
$$
Collecting these estimates, we get
\begin{multline*}
	\left|
	\sum\limits_{\substack{l\in\dualSU2\\ l\leq k}}(2l+1)\Tr[\widehat{f}(l)T^l(u)]
	\right|
\leq
	\sum\limits_{\substack{l\in\dualSU2\\ l\leq \xi}}(2l+1)^2\widehat{f}_l
+
	\frac{(2\xi+1)\widehat{f}_{\xi}}{t^2}
\\=
	\sum\limits_{\substack{l\in\dualSU2\\ l\leq \xi}}(2l+1)^2\widehat{f}_l
	+
	(2\xi+1)^3\widehat{f}_{\xi}\frac{(2\xi+1)}{(2\xi+1)^3}\frac1{t^2}.
\end{multline*}
By Theorem \ref{THM:pointwise-convergence-G} the partial sums $
	\sum\limits_{\substack{l\in\dualSU2\\ l\leq k}}(2l+1)\Tr[\widehat{f}(l)T^l(u)]
$
converge to $f(x)$ for almost all $x\in G$.
Then taking the limit as $k\to \infty$, we get
$$
|f(u)|
\lesssim
	\sum\limits_{\substack{l\in\dualSU2\\ l\leq \xi}}(2l+1)^2\widehat{f}_l
	+
	(2\xi+1)^3\widehat{f}_{\xi}\frac{(2\xi+1)}{(2\xi+1)^3}\frac1{t^2}.	
$$
We assumed that $(2l+1)\widehat{f}_l$ is a monotone sequence. Then $\widehat{f}_k$ is also a monotone 
decreasing sequence. Therefore, we get 
$$
	(2\xi+1)^3\widehat{f}_{\xi}
\leq 
	\sum\limits_{\substack{l\in\dualSU2\\ l\leq \xi}}(2l+1)^2\widehat{f}_l.
$$
Thus
\begin{multline*}
	\sum\limits_{\substack{l\in\dualSU2\\ l\leq \xi}}(2l+1)^2\widehat{f}_l
	+
	(2\xi+1)^3\widehat{f}_{\xi}\frac{2\xi+1}{(2\xi+1)^3}\frac1{t^2}
\leq
	\left(1+\frac{1}{(2\xi+1)^2}\frac1{t^2}\right)
	\sum\limits_{\substack{l\in\dualSU2\\ l\leq \xi}}(2l+1)^2\widehat{f}_l
\\\lesssim
\frac{1}{(2\xi+1)^2}\frac1{t^2}
	\sum\limits_{\substack{l\in\dualSU2\\ l\leq \xi}}(2l+1)^2\widehat{f}_l.
\end{multline*}
Since $\widehat{f}_l$ is {\it almost scalar}, by Definition \ref{DEF:almost_scalar}, the last sum equals to
$$
\frac{1}{(2\xi+1)^2}\frac1{t^2}
\sum\limits_{\substack{l\in\dualSU2\\ l\leq \xi}}(2l+1)\Tr \widehat{f}(l).
$$
Finally, we obtain
\begin{equation}\label{EQ:aux1}
|f(u)|
\lesssim
\frac{1}{(2\xi+1)^2}\frac1{t^2}
\sum\limits_{\substack{l\in\dualSU2\\ l\leq \xi}}(2l+1)\Tr \widehat{f}(l).
\end{equation}
This proves \eqref{EQ:integrability-criterion-1-p-2-proof-aux-1}.
Using this inequality and applying Weyl's integral formula for class functions (cf. e.g. Hall \cite{BrianHall2003}), 
we immediately get
\begin{multline*}
	\|f\|^p_{L^p(\SU2)}
=
	\int\limits_{[0,1]}
	|f(u)|^p\sin^2\frac{\pi t}2\,dt
\\
\lesssim
	\int\limits_{[0,1]}
	\left(
	\frac{1}{(2\xi+1)^2}\frac1{t^2}
	\sum\limits_{\substack{l\in\dualSU2\\ l\leq \xi}}(2l+1)\Tr \widehat{f}(l)	
	\right)^p
	\sin^2\frac{\pi t}2\,dt
.	
\end{multline*}
Here $\xi$ is an arbitrary fixed half-integer.
We split the interval $[0,1]$ as the union 
$[0,1]=\bigsqcup\limits_{\xi\in\dualSU2}[(2\xi+1+1)^{-1},(2\xi+1)^{-1}]$. 
Using the estimate with the corresponding $\xi$ in each interval of this decomposition, the last integral becomes
\begin{multline*}
\sum\limits_{\xi\in\dualSU2}
	\int\limits^{
				\frac1{(2\xi+1)}
				}_{
				\frac1{(2\xi+1+1)}
				}
	\left(
	\frac{1}{{(2\xi+1)}^2}\frac1{t^2}
	\sum\limits_{\substack{l\in\dualSU2\\ l\leq \xi}}
	(2l+1)\Tr\widehat{f}(l)
	\right)^p
	\sin^2\frac{\pi t}2\,dt
\\
\cong
\sum\limits_{\xi\in\dualSU2}
\int\limits^{
	\frac1{(2\xi+1)}
}_{
\frac1{(2\xi+1+1)
}}
\left(
\frac{1}{{(2\xi+1)}^2}\frac1{t^2}
\right)^p
\left(
\sum\limits_{\substack{l\in\dualSU2\\ l\leq \xi}}
(2l+1)\Tr\widehat{f}(l)
\right)^p
t^2\,dt.
\end{multline*}
Now, we notice that the inner sum $\sum\limits_{\substack{l\in\dualSU2\\ 2l+1\leq 2\xi+1}}
(2l+1)\Tr\widehat{f}(l)$ does not depend on $t$. Therefore, we can interchange summation and integration to get
\begin{multline*}
	\sum\limits_{\xi\in\dualSU2}
	\int\limits^{
		\frac1{(2\xi+1)}
	}_{
	\frac1{(2\xi+1+1)
	}}
	\left(
	\frac{1}{{(2\xi+1)}^2}\frac1{t^2}
	\right)^p
	\left(
		\sum\limits_{\substack{l\in\dualSU2\\ l\leq \xi}}
			(2l+1)\Tr\widehat{f}(l)				
	\right)^p
	t^2\,dt
\\	=
	\sum\limits_{\xi\in\dualSU2}
	\left(
	\frac{1}{{(2\xi+1)}^2}
	\right)^p
	\left(
	\sum\limits_{\substack{l\in\dualSU2\\ l\leq \xi}}
	(2l+1)\Tr\widehat{f}(l)		
	\right)^p
		\int\limits^{
			\frac1{(2\xi+1)}
		}_{
		\frac1{(2\xi+1+1)}}
		t^{2-2p}\,dt.
\end{multline*}
The key observation now is the fact that
$$
	\left(\frac{1}{(2\xi+1)^2}\right)^p
	\int\limits^{\frac1{(2\xi+1)}}_{\frac1{(2\xi+1+1)}}
	t^{2-2p}\,dt
\cong
(2\xi+1)^2(2\xi+1)^{3(p-2)}
\frac1{(2\xi+1)^{3p}}.
$$
Thus, the last sum, up to constant, equals to

$$
\sum\limits_{\xi\in\dualSU2}
(2\xi+1)^{-4}
\left(
\sum\limits_{\substack{l\in\dualSU2\\ l\leq \xi}}
(2l+1)\Tr\widehat{f}(l)
\right)^p.
$$
Thus, the last sum, up to constant, equals to
$$
\sum\limits_{\xi\in\dualSU2}
(2\xi+1)^2(2\xi+1)^{3(p-2)}
\left(
\frac1{(2\xi+1)^3}
\sum\limits_{\substack{l\in\dualSU2\\ 2l+1\leq 2\xi+1}}
(2l+1)\Tr\widehat{f}(l)
\right)^p.
$$
Now, we formulate and apply the following theorem proved by the authors in \cite{Npq:lattices}. 
Let $G$ be a compact Lie group and $\Gh$ its unitary dual. Let us denote by $\M_1$ the collection of all finite subsets $Q\subset\Gh$ of $\Gh$. Denote $\mu(Q)=\sum\limits_{\pi\in Q}d_\pi^2$ for $Q\in\M_1$.
\begin{thm}[\cite{Npq:lattices}] Let $1<p\leq 2$. Then we have
\begin{equation}
\label{necess_G-0}
\sum\limits_{\pi\in\Gh}d_\pi^2\jp\pi^{n(p-2)}\left(\sup_{\substack{Q\in\M_1\\\mu(Q)\geq \jp\pi^n}}\frac1{\mu(Q)}\left|\sum\limits_{\xi\in Q}d_{\xi}\Tr\widehat{f}(\xi)\right|\right)^p
=:
\|\widehat{f}\|_{N_{p',p}(\Gh,\M_1)}
\lesssim
\|f\|_{L^p(G)}.
\end{equation}
\end{thm}
Here $N_{p',p}(\Gh,\M_1)$ is the net space on the lattice $\Gh$ which has been discussed in \cite{Npq:lattices}.
For an arbitrary collection of finite subsets $M$,  in view of the embedding
(cf. \cite{Npq:lattices})
\begin{equation}
N_{p',p}(\Gh,\M)
\hookrightarrow
N_{p',p}(\Gh,\M_1)
\end{equation}
and inequality \eqref{necess_G-0},
 we get
\begin{equation}
\label{derived-necess-G-0}
\sum\limits_{\pi\in\Gh}d_\pi^2\jp\pi^{n(p-2)}\left(\sup_{\substack{Q\in\M\\\mu(Q)\geq \jp\pi^n}}\frac1{\mu(Q)}\left|\sum\limits_{\xi\in Q}d_{\xi}\Tr\widehat{f}(\xi)\right|\right)^p
\leq
\|f\|_{L^p(G)}.
\end{equation}
In particular, for $G=\SU2$ and $\M=\{\{\xi\in\Gh\colon \jp\xi\leq\jp\pi\}\colon \pi\in\Gh\}$, we thus obtain from 
\eqref{derived-necess-G-0} that
\begin{multline*}
	\sum\limits_{\xi\in\dualSU2}
	(2\xi+1)^2(2\xi+1)^{3(p-2)}
\left(
	\frac1{(2\xi+1)^3}
\sum_{\substack{l\in\dualSU2 \\ (2l+1)^3\leq 2\xi+1}}(2l+1)
\Tr\widehat{f}(l)
\right)^p
\\
\leq
	\sum\limits_{\xi\in\dualSU2}
	(2\xi+1)^2(2\xi+1)^{3(p-2)}
\left(
\sup_{\substack{k\in\dualSU2 \\ (2k+1)^3 \geq (2\xi+1)^3}}
	\frac1{(2k+1)^3}
\sum_{\substack{l\in\dualSU2 \\ (2l+1)^3\leq (2k+1)^3}}(2l+1)
\Tr\widehat{f}(l)
\right)^p
\\\leq
\|f\|^p_{L^p(\SU2)}.
\end{multline*}
This completes the proof.
\end{proof}

\subsection{Dirichlet kernel on $\SU2$}
\label{SEC:DK_estimate}

In the proof of Theorem \ref{THM:integrability-criterion-1-p-2-SU2} we made use of an estimate for the 
Dirichlet kernel on $\SU2$ which we now prove. We continue with the SU(2)-notation introduced in
\eqref{EQ:SU2-1}--\eqref{EQ:SU2-2}.

\begin{prop}
\label{PROP:DK_estimate}
On $\SU2$, the Dirichlet kernel
$$
	D_{l}(t):=
	\sum\limits_{\substack{k\in\dualSU2 \\ k\leq l}}(2k+1)\chi_k(t)=
	\sum\limits_{\substack{k\in\dualSU2 \\ k\leq l}}(2k+1)\frac{\sin(2k+1)\pi t}{\sin \pi t},\quad l\in\frac12\NN_0,
$$	
satisfies the estimate
\begin{equation}
\label{EQ:DK_estimate_SU2}
|D_l(t)|
\lesssim
\frac{2l+1}{t^2},
\end{equation}
with a constant independent of $t$ and $l$.
\end{prop}
\begin{proof}
%For $G=SU(2)\, |R_+|=1,\,\Delta(t)=\sin\frac{\pi t}2$, we have by direct calculations (see \cite{})
Since $\chi_k(t)=\Tr T^k(t)=\frac{\sin(2k+1)\pi t}{\sin \pi t}$, we have
$$
D_{l}(t)
=
\sum\limits_{\substack{k\in\dualSU2 \\ k\leq l}}(2k+1)\chi_k(t)
=
\sum\limits_{\substack{k\in\dualSU2 \\ k\leq l}}(2k+1)\frac{\sin(2k+1)\pi t}{\sin \pi t}.
$$
Using the fact that $\frac{d}{dt}\sin(2k+1)\pi t=(2k+1)\pi\cos(2k+1)\pi t$, we can represent the last sum as follows
\begin{multline*}
\frac1{\sin\pi t}
\sum\limits_{\substack{k\in\dualSU2 \\ k\leq l}}
(2k+1)\sin(2k+1)\pi t
=
\left(
\frac{-1}{\pi}
\right)
\frac1{ \sin\pi t}
\frac{d}{dt}
\left(
\sum\limits
_
{\substack{k\in\dualSU2 \\ k\leq l}}
\cos(2k+1)\pi t
\right)
\\=
\left(
\frac{-1}{\pi}
\right)
\frac1{\sin\pi t}
\frac{d}{dt}
\left(
\frac{
	\sum\limits
	_
	{\substack{k\in\dualSU2 \\ k\leq l}}
	\cos(2k+1)\pi t\sin \pi t}{\sin\pi t}
\right).
\end{multline*}
Using sine multiplication formula, we obtain
\begin{multline*}
\left(
\frac{-1}{\pi}
\right)
\frac1{\sin\pi t}
\frac{d}{dt}
\left(
\frac{\sum\limits_{\substack{k\in\dualSU2 \\ k\leq l}}
	\sin(2k+1+1)\pi t-\sin(2k+1-1)\pi t }{\sin\pi t}
\right)
\\=
\left(
\frac{-1}{\pi}
\right)
\frac1{\sin\pi t}
\frac{d}{dt}
\left(
\frac
{
\sin(2l+1)\pi t+\sin(2l+2)\pi t
}
{
\sin\pi t
}
\right)
\\=
\frac{
(
\sin(2l+1)\pi t+\sin(2l+2)\pi t
)
\cos(\pi t)
}{\sin^3\pi t}
\\
-\frac
{
(2l+1)\cos(2l+1)\pi t+(2l+2)\cos(2l+2)\pi t
}
{
\sin^2\pi t
}.
\end{multline*}

This proves \eqref{EQ:DK_estimate_SU2}.
\end{proof}
We can refer to Giulini and Travaglini \cite{Travaglini1980} and to Travaglini \cite{Travaglini1993} for some other interesting
properties of Fourier coefficients and Dirichlet kernels on $\SU2$.

\appendix
\section{Polyhedral summability on compact Lie groups}
\label{APP:polyhedral}

It has been shown by Stanton \cite{Sta1976} that for class functions on semisimple compact Lie groups
the polyhedral Fourier partial sums $S_N f$converge to $f$ in $L^p$ provided that 
$2-\frac1{s+1}<p<2+\frac1s$. Here the number $s$ depends on the root system
$\Rcal$ of the compact Lie group $G$, in the way we now describe.
We also note that the range of indices $p$ as above is sharp, see
Stanton and Tomas \cite{Stanton-Tomas:BAMS-1976,Stanton-Tomas:AJM-1978} as well as 
Colzani, Giulini and Travaglini \cite{CGT1989}. 
 
Let $G$ be a compact semisimple Lie group and let $T$ be a maximal torus of $G$, 
with Lie algebras $\mathfrak{g}$ and $\mathfrak{t}$, respectively. 
Let $n=\dim G$ and $l=\dim T=\rank G$. 
We define a positive definite inner product on $\mathfrak{t}$ by putting 
$(\cdot,\cdot)=-B(\cdot,\cdot)$, where $B$ is the Killing form.
Let $\Rcal$ be the set of roots of $\mathfrak{g}$. Choose in $\Rcal$ a system $\Rcal_+$ of positive roots (with cardinality $r$) and let $S=\{\alpha_1,\ldots,\alpha_{l}\}$ be the corresponding simple system. 
We define $\rho:=\frac12\sum\limits_{\alpha\in \Rcal_+}\alpha$.

For every 
$
\lambda\in{it}^*
$
there exists a unique 
$
H_{\lambda}\in\mathfrak{t}$ such that $\lambda(H)=i(H_{\lambda},H)
$
for every 
$H\in\mathfrak{t}$.
The vectors 
$
H_{j}=\frac{4\pi iH_{\alpha_j}}{\alpha_j(H_{\alpha_j})}
$ 
generate the lattice sometimes denoted by ${\rm Ker(exp)}$.
The elements of the set 
$$
\Lambda
=
\{
\lambda\in{it}^*\colon \lambda(H)\in 2\pi i\ZZ,\;\textrm{ for any }\; H\in {\rm Ker(exp)}
\}
$$ 
are called the weights of $G$ and the fundamental weights are defined by the relations 
$\lambda_j=2\pi i\delta_{jk}$, $j,k=1,\ldots,l$.
The subset 
$$
\mathfrak{D}
=
\{
\lambda\in\Lambda\colon \lambda=\sum^l_{j=1}m_j\lambda_j,\,m_j\in\NN
\}
$$
of the set $\Lambda$ with positive coordinates $m_j$ is called the set of dominant weights. Here, the word `dominant' means that with respect to a certain partial order on the set $\Lambda$ every weight $\lambda=\sum^l_{j=1}m_j\lambda_j$ with $m_j>0$ is maximal.
There exists a bijection between $\Gh$ and the semilattice $\mathfrak{D}$ of the dominant weights of $G$, i.e.
$$
\mathfrak{D}\ni\lambda=(m_1,\ldots,m_{l})\longleftrightarrow
\pi\in\Gh.
$$ 
Therefore, we will not distinguish between $\pi$ and the corresponding dominant weight $\lambda$ and will write 
\begin{equation}
\label{EQ:definition_explaination}
\pi=(\pi_1,\ldots,\pi_{l}),
\end{equation}
where we agree to set $\pi_i=m_i$.
%We denote by $L^*_p(G)$ the Banach space $L_p$ of central functions on $G$.
With $\rho=\frac12\sum\limits_{\alpha\in \Rcal_+}\alpha$, for a natural number $N\in\NN$, we set
\begin{equation}\label{EQ:QN}
Q_{N}:=\{\xi\in\Gh \colon \; \xi_i\leq N\rho_i, \; i=1,\ldots,l\}.
\end{equation}
We call $Q_{N}$ a finite polyhedron of $N^{\rm th}$ order and denote by $\M_0$ the set of all finite polyhedrons in $\Gh$ or in $\Gh_0$.  

Now, fix an arbitrary fundamental weight $\lambda_j$, $j=1,\ldots,l$,
and set 
$
\Rcal_{\lambda_j}^{\perp}:=\{\alpha \in \Rcal_+ \colon (\alpha,\lambda_j)=0\},$
and $\Rcal_+=\Rcal_{\lambda_j}\oplus \Rcal_{\lambda_j}^{\perp}$. 
We will often use the number
\begin{equation}
\label{EQ:s}
s:=\max\limits_{\substack{j=1,\ldots,l}} \card \Rcal_{\lambda_j}.
\end{equation}
We denote by $L^{p}_{*}(G/K)$ the Banach subspace of $L^p(G/K)$ of functions on $G/K$ whose canonical liftings
are central on $G$: if $\tilde f(g)=f(gK)$ is the canonical lifting of $f$ from $G/K$ to $G$, by definition
$$
	f\in L^{p}_{*}(G/K) \text{ if and only if } f\in L^p(G/K) \text{ and } \tilde f(gug^{-1})=\tilde f(u) \quad \text{ for all } u,g\in G.
$$
We note that such functions have then $K$-invariance both on the right and on the left:
$\tilde f(K u K)=\tilde f(u)$ for all $u\in G$. Consequently, for $\pi\in\Gh_0$, 
with our choice of basis vectors for the invariant subspace of the representation space,
the Fourier coefficient $\widehat{f}(\pi)$ vanishes outside the upper-left $\kpi\times \kpi$ block, i.e.
$\widehat{f}(\pi)_{ij}=0$ if $i>\kpi$ or $j>\kpi$.

Further, we formulate and apply a result on semisimple Lie groups by Robert Stanton \cite{Sta1976} 
for $L^p$-norm convergence of polyhedral Fourier partial sums.
We also refer to Stanton and Tomas \cite{Stanton-Tomas:BAMS-1976,Stanton-Tomas:AJM-1978} and 
to Colzani, Giulini and Travaglini \cite{CGT1989} for the converse statement.

Let $\rho$ denote the half-sum of positive roots of $G$.  
Recall also the notation 
$Q_N:=\{\pi\in\Gh_0 \colon \pi_i \leq N\rho_i,\; i=1,\ldots,l\}$ and $D_N(u):=D_{Q_N}(u)=\sum\limits_{\pi\in Q_N}\Tr[\pi(u)]$. 

\begin{thm}[\cite{Sta1976}]\label{THM:polyhedral_mean_convergence}
Let $G$ be a semisimple compact Lie group.
	Let $f\in L^p_*(G/K)$ and let $S_Nf(x)$ be the associated polyhedral Fourier partial sum, i.e.
	$$
	S_Nf(u):=T_{Q_N}(x).
	$$
	
Then $S_Nf$	converges to $f$ in $L^p(G/K)$ provided that $2-\frac1{1+s}<p<2+\frac1s$,
where $s$ is defined by \eqref{EQ:s}. If $G$ is simply connected, this range of $p$ is in general sharp.
\end{thm}

Consequently, one obtains

\begin{thm}
\label{THM:pointwise-convergence-G}
 Let $\frac{2n}{n+l}<p<+\infty$ and $f\in L^p_*(G)$. Then $S_Nf(x)$ converges to $f(x)$ for almost all $x\in G$.
\end{thm}
Although Stanton's version of this theorem is on groups, by considering the canonical liftings from the
homogeneous space we obtain the formulation above also for homogeneous spaces, at least for
the sufficient condition. The only if part of  `in general sharp' follows from 
\cite{CGT1989}, by for example taking $K=\{e\}$.

\section{Marcinkiewicz interpolation theorem}
\label{SEC:Marc_Interpol_Theorem}

In this section we formulate the Marcinkiewicz interpolation theorem on arbitrary $\sigma$-finite measure spaces.
Then we show how to use this theorem for linear mappings between $C^{\infty}(G)$ and the space $\Sigma$ of finite matrices on the discrete unitary dual $\Gh$ or on the discrete set $\Gh_0$ of class I representations with different measures on $\Gh$ and $\Gh_0$.

%Then, putting $X=G, X=G/K$ and $\Gamma=\Gh,\Gamma=\Gh_0$, we apply this theorem to linear mappings between functions and finite matrices on $(X,\mu_{X})$ and the space $\Sigma$ of finite matrices on $(\Gamma,\nu_{\Gamma}$. Then, putting $\Gamma=\Gh$ or $\Gamma=\Gh_0$, we apply this version of the Marcinkewicz interpolation theorem with different choice of measures on the discrete unitary dual $\Gh$ and $\Gh_0$ respectively.
%
%
%In this section we prove an abstract version of the Marcinkiewicz interpolation theorem 
%to be able to apply it with different measures on the discrete unitary dual $\Gh$ and on the discrete set $\Gh_0$
%of class I representations. 

This approach will be instrumental in the proof
of the Hardy-Littlewood Theorem \ref{HL-groups-1} and of the Paley inequality in
Theorem \ref{THM:Paley_inequality}.

We now formulate the Marcinkiewicz theorem for linear mappings between functions on arbitrary $\sigma$-finite measure spaces $(X,\mu_{X})$ and $(\Gamma,\nu_{\Gamma})$.

Let $PC(X)$ denote the space of step functions on $(X,\mu_{X})$. 
%$$
%PC(X)\ni f \mapsto Af=h=\{h(\pi)\}_{\pi\in\Gamma},
%$$
%In this section we formulate and prove Marcinkiewicz interpolation theorem for linear 
%mappings between a measure space $(X,\mu)$ and the space  $\Sigma$ on $(\Gamma,\nu)$.
%Thus, a linear mapping $A\colon PC(X) \to \Sigma$ takes a function to a matrix valued sequence, 
%i.e.
%$$
%PC(X)\ni f \mapsto Af=h=\{h(\pi)\}_{\pi\in\Gamma},
%$$
%where
%$$
%h(\pi)\in \mathbb{C}^{\kappa_{\pi}\times \delta_{\pi}},\; \pi\in\Gamma,
%$$
%and
%$$
%PC(X)=\{f\colon X \to \RR \colon \text{ $f$ is a step function}\}.
%$$
We say that a linear operator $A$ is of strong type $(p,q)$, if for every $f\in L^p(X,\mu_{X})\cap PC(X)$, 
we have $Af\in L^q(\Gamma,\nu_{\Gamma})$ and
$$
\|Af\|_{L^q(\Gamma,\nu_{\Gamma})}
\leq 
C
\|f\|_{L^p(X, \mu_{X})},
$$
where $C$ is independent of $f$, and the space $\ell^q(\Gamma,\nu_{\Gamma})$ defined by the norm
\begin{equation}
\|h\|_{L^q(\Gamma,\nu_{\Gamma})}
:=
\left(
\int\limits_{\Gamma}
|h(\pi)|^p
\nu(\pi)
\right)^{\frac1q}.
\end{equation}
The least $C$ for which this is satisfied is taken
to be the strong $(p,q)$-norm of the operator $A$.

Denote the distribution functions of $f$ and $h$ by $\mu_{X}(x;f)$ and $\nu_{\Gamma}(y;h)$, 
respectively, i.e.
\begin{eqnarray}\nonumber
\mu_{X}(x;f)
:=
\int\limits_{\substack{t\in X \\ |f(t)|\geq x}}d\mu(t),\quad x>0,\\
\nu_{\Gamma}(y;h)
:=
\int\limits_{\substack{
		\pi\in\Gamma
		\\ 
		|h(\pi)|
		\geq y
	}}d\nu(\pi),\quad y>0.
	\label{EQ:density2}
	\end{eqnarray}
	Then
	$$
	\begin{aligned}
	\|f\|^p_{L^p(X,\mu_{X})}
	&=
	\int\limits_{X}
	|f(t)|^p\,d\mu(t)
	=
	p\int\limits^{+\infty}_{0}x^{p-1}\mu_{X}(x;f)\,dx,
	\\
	\|h\|^q_{L^q(\Gamma,\nu_{\Gamma})}
	&=
	\int\limits_{\pi\in\Gamma}
	|h(\pi)|^q
	\nu(\pi)
	=
	q\int\limits^{+\infty}_{0}y^{q-1}\nu_{\Gamma}(y;h)\,dy.
	\end{aligned}
	$$
	A linear operator $A\colon \mathcal PC(X) \rightarrow L^q(\Gamma,\nu_{\Gamma})$ satisfying
	\begin{equation}
	\label{EQ:weak_type}
	\nu_{\Gamma}(y;Af)
	\leq
	\left(
	\frac{M}{y}\|f\|_{L^p(X,\mu_{X})}
	\right)^{q},\quad \text{for any $y>0$.}
	\end{equation}
	is said to be of {\it weak type} $(p,q)$; the least value of $M$ in \eqref{EQ:weak_type} is called the weak $(p,q)$ norm of $A$.
	
	Every operation of strong type $(p,q)$ is also of weak type $(p,q)$, since
	$$
	y\left(
	\nu_{\Gamma}(y;Af)
	\right)^{\frac1q}
	\leq
	\|Af\|_{L^q(\Gamma)}
	\leq
	M
	\|f\|_{L^p(X)}.
	$$
	\begin{thm} 
		\label{THM:Marc-Interpol}
		Let $1\leq p_1<p<p_2<\infty$. Suppose that a linear operator $A$ 
		from $\mathcal PC(X)$ to $L^q(\Gamma,\nu_{\Gamma})$
		is simultaneously 
		of {\it weak types} $(p_1,p_1)$ and $(p_2,p_2)$, with norms $M_1$ and $M_2$,
		respectively, i.e.
		\begin{eqnarray*}
		%\label{EQ:weak_type_1}
		\nu_{\Gamma}(y;Af)
		&
		\leq
		&
		\left(
		\frac{M_1}{y}\|f\|_{L^{p_1}(X,\mu_{X})}
		\right)^{p_1},
		\\
%		\label{EQ:weak_type_2}
		\nu_{\Gamma}(y;Af)
		&
		\leq
		&
		\left(
		\frac{M_2}{y}\|f\|_{L^{p_2}(X,\mu_{X})}
		\right)^{p_2}\quad \text{ hold for any $y>0$.}
		\end{eqnarray*}
		Then for any $p\in(p_1,p_2)$ the operator $A$ is of strong type $(p,p)$ and we have
		\begin{equation*}
%		\label{EQ:norm_interpolation_estimate}
		\|Af\|_{L^p(\Gamma,\nu_{\Gamma})}
		\lesssim
		M^{1-\theta}_1M^{\theta}_2\|f\|_{L^p(X,\mu_{X})},\quad 0<\theta<1,
		\end{equation*}
		where
		\begin{eqnarray*}
			%\frac1q=\frac{1-\theta}{q_1}+\frac{\theta}{q_2},\\
			\frac1p=\frac{1-\theta}{p_1}+\frac{\theta}{p_2}.
		\end{eqnarray*}
	\end{thm}
	The proof is given in e.g. Folland \cite{Folland1999}.
Now, we adapt this theorem to the setting of matrix-valued mappings. 

Suppose $\Gamma$ is a discrete set.  Integral over $\Gamma$ is defined as sum over $\Gamma$, i.e.
\begin{equation}
\label{EQ:integration_over_discrete_set}
\int\limits_{\Gamma}\nu_{\Gamma}(\pi)
:=
\sum\limits_{\pi\in\Gamma}\nu(\pi).
\end{equation}
In this case, to define a measure on $\Gamma$ means to define a real-valued positive sequence $\nu=\{\nu_{\pi}\}_{\pi\in\Gamma}$, i.e.
$$
\Gamma\ni \pi \mapsto \nu_{\pi}\in\RR_+.
$$
We turn $\Gamma$ into a $\sigma$-finite measure space by introducing a measure 
$$
\nu_{\Gamma}(Q):=\sum\limits_{\pi\in Q}\nu_{\pi},
$$
where $Q$ is arbitrary  subset of $\Gamma$. 

We consider two sequences $\delta=\{\delta_{\pi}\}_{\pi\in\Gamma}$ and $\kappa=\{\kappa_{\pi}\}_{\pi\in\Gamma}$, i.e.
\begin{eqnarray*}
\Gamma\ni \pi \mapsto \delta_{\pi}\in\NN,
\\
\Gamma\ni \pi \mapsto \kappa_{\pi}\in\NN.
\end{eqnarray*}
We denote by $\Sigma$ the space of matrix-valued sequences on $\Gamma$ that will be realised via
$$\Sigma:=\left\{ h=\{h(\pi)\}_{\pi\in\Gamma}, 
h(\pi)\in \mathbb{C}^{\kappa_{\pi}\times \delta_{\pi}}\right\}.$$
The $\ell^p$ spaces on $\Sigma$ can be defined, for example, 
motivated by the Fourier analysis on compact homogeneous spaces, in the form 
$$
\|h\|_{\ell^p(\Gamma,\Sigma)}
:=
\left(
\sum\limits_{\pi\in\Gamma}
\left(
\frac{\|h(\pi)\|_{\HS}}{\sqrt{\kpi}}
\right)^p
\nu_{\pi}
\right)^{\frac1p},\quad h\in\Sigma.
$$
If we put $X=G$, where $G$ is a compact Lie group and let $\Gamma=\Gh$, then Fourier transform can be regarded as an operator mapping a function $f\in L^p(G)$ to the matrix-valued sequence $\widehat{f}=\{\widehat{f}(\pi)\}_{\pi\in\Gh}$  of the Fourier coefficients,
with $\delta_\pi=\kappa_\pi=\dpi$.
For $\Gamma=\Gh_0$ we put $\delta_{\pi}=\dpi$ and $\kappa_{\pi}=\kpi$,  these spaces thus coincide with the $\ell^p(\Gh_0)$ spaces introduced in \cite{RT}. 
In Section \ref{SEC:proofs}, choosing different measures $\{\nu_{\pi}\}_{\pi\in\Gamma}$ on the unitary dual $\Gh$ or on the set $\Gh_0$, we use this to prove the Paley inequality and Hausdorff-Young-Paley inequalitites.
Let us denote by $|h|$ the sequence consisting of $\{\frac{\|h(\pi)\|_{\HS}}{\sqrt{\kpi}}\}$, i.e.
$$
|h|
=
\left\{\frac{\|h(\pi)\|_{\HS}}{\sqrt{\kpi}}\right\}_{\pi\in\Gamma}.
$$
Then, we have
$$
\|h\|_{\ell^q(\Gamma,\Sigma)}
=
\||h|\|_{L^q(\Gamma,\nu_{\Gamma})}.
$$
Thus, we obtain
	\begin{thm} 
		\label{THM:Marc-Interpol-Lattice}
		Let $1\leq p_1<p<p_2<\infty$. Suppose that a linear operator $A$ 
		from $\mathcal PC(X)$ to $\Sigma$
		is simultaneously 
		of {\it weak types} $(p_1,p_1)$ and $(p_2,p_2)$, with norms $M_1$ and $M_2$,
		respectively, i.e.
		\begin{eqnarray}
		\label{EQ:weak_type_1}
		\nu_{\Gamma}(y;Af)
		&
		\leq
		&
		\left(
		\frac{M_1}{y}\|f\|_{L^{p_1}(X)}
		\right)^{p_1},
		\\
		\label{EQ:weak_type_2}
		\nu_{\Gamma}(y;Af)
		&
		\leq
		&
		\left(
		\frac{M_2}{y}\|f\|_{L^{p_2}(X)}
		\right)^{p_2}\quad \text{ hold for any $y>0$.}
		\end{eqnarray}
		Then for any $p\in(p_1,p_2)$ the operator $A$ is of strong type $(p,p)$ and we have
		\begin{equation}
		\label{EQ:norm_interpolation_estimate}
		\|Af\|_{\ell^p(\Gamma,\Sigma)}
		\leq
		M^{1-\theta}_1M^{\theta}_2\|f\|_{L^p(X)},\quad 0<\theta<1,
		\end{equation}
		where
		\begin{eqnarray*}
			%\frac1q=\frac{1-\theta}{q_1}+\frac{\theta}{q_2},\\
			\frac1p=\frac{1-\theta}{p_1}+\frac{\theta}{p_2}.
		\end{eqnarray*}
	\end{thm}

\def\germ{\frak} \def\scr{\cal} \ifx\documentclass\undefinedcs
  \def\bf{\fam\bffam\tenbf}\def\rm{\fam0\tenrm}\fi % f**k-amstex!
  \def\defaultdefine#1#2{\expandafter\ifx\csname#1\endcsname\relax
  \expandafter\def\csname#1\endcsname{#2}\fi} \defaultdefine{Bbb}{\bf}
  \defaultdefine{frak}{\bf} \defaultdefine{=}{\B} % doublef**k-amstex!!
  \defaultdefine{mathfrak}{\frak} \defaultdefine{mathbb}{\bf}
  \defaultdefine{mathcal}{\cal}
  \defaultdefine{beth}{BETH}\defaultdefine{cal}{\bf} \def\bbfI{{\Bbb I}}
  \def\mbox{\hbox} \def\text{\hbox} \def\om{\omega} \def\Cal#1{{\bf #1}}
  \def\pcf{pcf} \defaultdefine{cf}{cf} \defaultdefine{reals}{{\Bbb R}}
  \defaultdefine{real}{{\Bbb R}} \def\restriction{{|}} \def\club{CLUB}
  \def\w{\omega} \def\exist{\exists} \def\se{{\germ se}} \def\bb{{\bf b}}
  \def\equivalence{\equiv} \let\lt< \let\gt>


\begin{thebibliography}{CW71b}

\bibitem[ANR14]{HLP}
R.~Akylzhanov, E.~Nursultanov, and M.~Ruzhansky.
\newblock {H}ardy-{L}ittlewood-{P}aley inequalities and {F}ourier multipliers
  on {SU}(2).
\newblock {\em arXiv:1403.1731}, 2014.

\bibitem[ANR15]{Npq:lattices}
R.~Akylzhanov, E.~Nursultanov, and M.~Ruzhansky.
\newblock Net spaces on lattices.
\newblock {\em preprint}, 2015.

\bibitem[BL76]{BL2011}
J.~Bergh and J.~L{{\"o}}fstr{{\"o}}m.
\newblock {\em Interpolation spaces. {A}n introduction}.
\newblock Springer-Verlag, Berlin-New York, 1976.
\newblock Grundlehren der Mathematischen Wissenschaften, No. 223.

\bibitem[CdG71]{Coifman-deGuzman:SU2-Argentina-1970}
R.~R. Coifman and M.~de~Guzm{{\'a}}n.
\newblock Singular integrals and multipliers on homogeneous spaces.
\newblock {\em Rev. Un. Mat. Argentina}, 25:137--143, 1970/71.
\newblock Collection of articles dedicated to Alberto Gonz{{\'a}}lez
  Dom{\'{\i}}nguez on his sixty-fifth birthday.

\bibitem[CGT89]{CGT1989}
L.~Colzani, S.~Giulini, and G.~Travaglini.
\newblock Sharp results for the mean summability of {F}ourier series on compact
  {L}ie groups.
\newblock {\em Math. Ann.}, 285(1):75--84, 1989.

\bibitem[Cle71]{Clerc:SU2-CRAS-1971}
J.-L. Clerc.
\newblock Fonctions de {P}aley-{L}ittlewood sur {${\rm SU}(2)$} attach{\'e}es
  aux sommes de {R}iesz.
\newblock {\em C. R. Acad. Sci. Paris S{\'e}r. A-B}, 272:A1697--A1699, 1971.

\bibitem[Cle72]{Clerc:spherical-sums-CRAS-1972}
J.-L. Clerc.
\newblock Les sommes partielles de la d{\'e}composition en harmoniques
  sph{\'e}riques ne convergent pas dans {${\bf {\rm }L}^{p}$} {$(p\not=2)$}.
\newblock {\em C. R. Acad. Sci. Paris S{\'e}r. A-B}, 274:A59--A61, 1972.

\bibitem[Cle73]{Clerc:PhD-thesis}
J.~L. Clerc.
\newblock {\em Th{\`e}se}.
\newblock PhD thesis, Universit\'e Paris XI, 1973.

\bibitem[CW71a]{coifman+weiss_lnm}
R.~R. Coifman and G.~Weiss.
\newblock {\em Analyse harmonique non-commutative sur certains espaces
  homog{\`e}nes}.
\newblock Lecture Notes in Mathematics, Vol. 242. Springer-Verlag, Berlin,
  1971.
\newblock {{\'E}}tude de certaines int{{\'e}}grales singuli{{\`e}}res.

\bibitem[CW71b]{Coifman-Weiss:SU2-Argentina-1970}
R.~R. Coifman and G.~Weiss.
\newblock Multiplier transformations of functions on {${\rm SU}(2)$} and {$\sum
  _2$}.
\newblock {\em Rev. Un. Mat. Argentina}, 25:145--166, 1971.
\newblock Collection of articles dedicated to Alberto Gonz{{\'a}}lez
  Dom{\'{\i}}nguez on his sixty-fifth birthday.

\bibitem[DR14]{Dasgupta-Ruzhansky:Gevrey-BSM}
A.~Dasgupta and M.~Ruzhansky.
\newblock Gevrey functions and ultradistributions on compact {L}ie groups and
  homogeneous spaces.
\newblock {\em Bull. Sci. Math.}, 138(6):756--782, 2014.

\bibitem[Edw72]{Edwards:BK}
R.~E. Edwards.
\newblock {\em Integration and harmonic analysis on compact groups}.
\newblock Cambridge Univ. Press, London, 1972.
\newblock London Mathematical Society Lecture Note Series, No. 8.

\bibitem[Fef71a]{Fefferman:ball-AM-1971}
C.~Fefferman.
\newblock The multiplier problem for the ball.
\newblock {\em Ann. of Math. (2)}, 94:330--336, 1971.

\bibitem[Fef71b]{Fefferman:divergence-BAMS-1971}
C.~Fefferman.
\newblock On the divergence of multiple {F}ourier series.
\newblock {\em Bull. Amer. Math. Soc.}, 77:191--195, 1971.

\bibitem[Fol99]{Folland1999}
G.~B. Folland.
\newblock {\em Real analysis}.
\newblock Pure and Applied Mathematics (New York). John Wiley \& Sons, Inc.,
  New York, second edition, 1999.
\newblock Modern techniques and their applications, A Wiley-Interscience
  Publication.

\bibitem[GT80]{Travaglini1980}
S.~Giulini and G.~Travaglini.
\newblock {$L^{p}$}-estimates for matrix coefficients of irreducible
  representations of compact groups.
\newblock {\em Proc. Amer. Math. Soc.}, 80(3):448--450, 1980.

\bibitem[Hal03]{BrianHall2003}
B.~C. Hall.
\newblock {\em Lie groups, Lie algebras, and Representations. An elementary
  introduction}.
\newblock Graduate Texts in Mathematics. Springer-Verlag, 2003.

\bibitem[Her54]{Herz:PNAS-1954}
C.~S. Herz.
\newblock On the mean inversion of {F}ourier and {H}ankel transforms.
\newblock {\em Proc. Nat. Acad. Sci. U. S. A.}, 40:996--999, 1954.

\bibitem[HL27]{HL}
G.~H. Hardy and J.~E. Littlewood.
\newblock Some new properties of {F}ourier constants.
\newblock {\em Math. Ann.}, 97(1):159--209, 1927.

\bibitem[H{\"{o}}r60]{Hormander:invariant-LP-Acta-1960}
L.~H{\"{o}}rmander.
\newblock Estimates for translation invariant operators in {$L^{p}$}\ spaces.
\newblock {\em Acta Math.}, 104:93--140, 1960.

\bibitem[HR70]{Hewitt-Ross:BK-Vol-II}
E.~Hewitt and K.~A. Ross.
\newblock {\em Abstract harmonic analysis. {V}ol. {II}: {S}tructure and
  analysis for compact groups. {A}nalysis on locally compact {A}belian groups}.
\newblock Die Grundlehren der mathematischen Wissenschaften, Band 152.
  Springer-Verlag, New York, 1970.

\bibitem[HR74]{HR}
E.~Hewitt and K.~A. Ross.
\newblock {Rearrangements of $L^r$ Fourier series on compact Abelian groups.}
\newblock {\em {Proc. Lond. Math. Soc. (3)}}, 29:317--330, 1974.

\bibitem[Kun58]{Kunze:FT-TAMS-1958}
R.~A. Kunze.
\newblock {$L_{p}$} {F}ourier transforms on locally compact unimodular groups.
\newblock {\em Trans. Amer. Math. Soc.}, 89:519--540, 1958.

\bibitem[NRT14]{NRT2014}
E.~Nursultanov, M.~Ruzhansky, and S.~Tikhonov.
\newblock Nikolskii inequality and {B}esov, {T}riebel-{L}izorkin, {W}iener and
  {B}eurling spaces on compact homogeneous manifolds.
\newblock {\em http://arxiv.org/abs/1403.3430}, 2014.

\bibitem[NT00]{Nurs_Tleukh_Mult}
E.~D. Nursultanov and N.~T. Tleukhanova.
\newblock Lower and upper bounds for the norm of multipliers of multiple
  trigonometric {F}ourier series in {L}ebesgue spaces.
\newblock {\em Funktsional. Anal. i Prilozhen.}, 34(2):86--88, 2000.

\bibitem[Pes08]{Pesenson:Besov-2008}
I.~Pesenson.
\newblock Bernstein-{N}ikolskii inequalities and {R}iesz interpolation formula
  on compact homogeneous manifolds.
\newblock {\em J. Approx. Theory}, 150(2):175--198, 2008.

\bibitem[RT10]{RT}
M.~Ruzhansky and V.~Turunen.
\newblock {\em Pseudo-differential operators and symmetries. Background
  analysis and advanced topics}, volume~2 of {\em Pseudo-Differential
  Operators. Theory and Applications}.
\newblock Birkh{\"a}user Verlag, Basel, 2010.

\bibitem[RT13]{Ruzhansky+Turunen-IMRN}
M.~Ruzhansky and V.~Turunen.
\newblock Global quantization of pseudo-differential operators on compact {L}ie
  groups, {$\rm SU(2)$}, 3-sphere, and homogeneous spaces.
\newblock {\em Int. Math. Res. Not. IMRN}, (11):2439--2496, 2013.

\bibitem[RW13]{RuWi2013}
M.~Ruzhansky and J.~Wirth.
\newblock On multipliers on compact {L}ie groups.
\newblock {\em Funct. Anal. Appl.}, 47(1):87--91, 2013.

\bibitem[RW15]{Ruzhansky-Wirth-multipliers}
M.~Ruzhansky and J.~Wirth.
\newblock ${L}^p$ {F}ourier multipliers on compact {L}ie groups.
\newblock {\em Math. Z., http://doi.org/10.1007/s00209-015-1440-9}, 2015.

\bibitem[Shu87]{Shubin:BK-1987}
M.~A. Shubin.
\newblock {\em Pseudodifferential operators and spectral theory}.
\newblock Springer Series in Soviet Mathematics. Springer-Verlag, Berlin, 1987.
\newblock Translated from the Russian by Stig I. Andersson.

\bibitem[ST76]{Stanton-Tomas:BAMS-1976}
R.~J. Stanton and P.~A. Tomas.
\newblock Convergence of {F}ourier series on compact {L}ie groups.
\newblock {\em Bull. Amer. Math. Soc.}, 82(1):61--62, 1976.

\bibitem[ST78]{Stanton-Tomas:AJM-1978}
R.~J. Stanton and P.~A. Tomas.
\newblock Polyhedral summability of {F}ourier series on compact {L}ie groups.
\newblock {\em Amer. J. Math.}, 100(3):477--493, 1978.

\bibitem[Sta76]{Sta1976}
R.~J. Stanton.
\newblock Mean convergence of {F}ourier series on compact {L}ie groups.
\newblock {\em Trans. Amer. Math. Soc.}, 218:61--87, 1976.

\bibitem[Ste70]{Stein:BOOK-topics-Littlewood-Paley}
E.~M. Stein.
\newblock {\em Topics in harmonic analysis related to the {L}ittlewood-{P}aley
  theory.}
\newblock Annals of Mathematics Studies, No. 63. Princeton University Press,
  Princeton, N.J., 1970.

\bibitem[Tra93]{Travaglini1993}
G.~Travaglini.
\newblock Polyhedral summability of multiple {F}ourier series (and explicit
  formulas for {D}irichlet kernels on {${\bf T}^n$} and on compact {L}ie
  groups).
\newblock {\em Colloq. Math.}, 65(1):103--116, 1993.

\bibitem[Vil68]{Vilenkin:BK-eng-1968}
N.~J. Vilenkin.
\newblock {\em Special functions and the theory of group representations}.
\newblock Translated from the Russian by V. N. Singh. Translations of
  Mathematical Monographs, Vol. 22. American Mathematical Society, Providence,
  R. I., 1968.

\bibitem[VK91]{VK1991}
N.~J. Vilenkin and A.~U. Klimyk.
\newblock {\em Representation of {L}ie groups and special functions. {V}ol. 1},
  volume~72 of {\em Mathematics and its Applications (Soviet Series)}.
\newblock Kluwer Academic Publishers Group, Dordrecht, 1991.
\newblock Simplest Lie groups, special functions and integral transforms,
  Translated from the Russian by V. A. Groza and A. A. Groza.

\end{thebibliography}
\end{document}